\documentclass[11pt]{amsart}
\usepackage{geometry}
\usepackage{url,amssymb,amsmath,amsthm}
\usepackage[pdftex]{graphicx}
\usepackage[all]{xy}
\usepackage{ stmaryrd }
\usepackage{tikz}
\usetikzlibrary{matrix,arrows}
\title[Galois cohomology, ramification and Wieferich primes]{Wieferich primes and a mod $p$ Leopoldt conjecture}
\author[G. B\"ockle, D.-A. Guiraud,  S. Kalyanswamy, C. Khare]{Gebhard B\"ockle, David-A. Guiraud,  Sudesh Kalyanswamy, Chandrashekhar Khare}
\address{Gebhard B\"ockle, IWR, Universit\"at Heidelberg, 69120 Heidelberg, Germany}
\email{gebhard.boeckle@iwr.uni-heidelberg.de}
\address{David-A. Guiraud,  IWR, Universit\"at Heidelberg, 69120 Heidelberg, Germany}
\email{david-alexandre.guiraud@iwr.uni-heidelberg.de\\}
\address{Sudesh Kalyanswamy, Yale University, Mathematics Dept., New Haven, CT 06520-8283, USA}
\email{sudesh.kalyanswamy@yale.edu}
\address{Chandrashekhar B.  Khare, Dept. of Mathematics, UCLA, Los Angeles, CA 90095-1555 , USA}
\email{shekhar84112@gmail.com}

\thanks{G.B. and D.-A.G. were supported by the DFG program FG 1920. C.K. was supported by NSF grant DMS - 1601692 and a Humboldt Research Award. }
\date{\today}

\newcommand{\id}{\mathrm{id}}

\DeclareMathOperator{\Hom}{Hom}
\DeclareMathOperator{\End}{End}
\DeclareMathOperator{\Ad}{Ad}
\DeclareMathOperator{\Ind}{Ind}

\DeclareMathOperator{\im}{im}
\DeclareMathOperator{\Frob}{Frob}
\newcommand{\ord}{\textrm{ord}}
\newcommand{\unr}{\textrm{unr}}
\newcommand{\qnew}{{q\textrm{-new}}}
\newcommand{\qnewtw}{{q\textrm{-new-tw}}}
\newcommand{\qunr}{{q\textrm{-unr}}}
\DeclareMathOperator{\GL}{GL}
\DeclareMathOperator{\SL}{SL}

\DeclareMathOperator{\CNL}{CNL}
\DeclareMathOperator{\Aut}{Aut}

\DeclareMathOperator{\Gal}{Gal}

\DeclareMathOperator{\univ}{univ}

\newcommand{\cO}{{\mathcal O}}

\newcommand{\ffrm}{{\mathfrak m}}

\newcommand{\frp}{{\mathfrak p}}

\newcommand{\bbF}{{\mathbb F}}

\newcommand{\bbQ}{{\mathbb Q}}

\newcommand{\T}{\mathbb{T}}
\newcommand{\A}{\mathbb{A}}
\newcommand{\R}{\mathbb{R}}

\newcommand{\C}{\mathbb{C}}

\newcommand{\Z}{\mathbb{Z}}

\newcommand{\Q}[0]
{
\mathbb{Q}
}

\def\eps{\epsilon}
\def\rhobar{ {\overline{\rho}} }

\newcommand{\ra}{\rightarrow}

\newcommand{\F}[0]
{
\mathbb{F}
}
\advance\textheight by 2.66em
\advance\topmargin by -.66em

\begin{document}
\begin{abstract} We consider questions in Galois cohomology which arise by considering mod $p$ Galois representations arising from automorphic forms. We consider a Galois cohomological analog for the standard heuristics about  the distribution
of Wieferich primes, i.e. prime $p$ such that $2^{p-1}$ is 1 mod $p^2$. Our analog relates to asking if in a compatible system of Galois representations, for almost all primes $p$, the residual mod $p$ representation arising from it has unobstructed deformation theory.   This analog leads in particular  to formulating a mod $p$ analog  for almost all primes $p$ of the  classical Leopoldt conjecture, which has been considered previously  by G. Gras. Leopoldt conjectured that  for a number field $F$,  and a prime $p$,  the $p$-adic regulator $R_{F,p}$ is non-zero.  The  mod $p$ analog is that  for a fixed number field $F$, for almost all primes $p$,  the $p$-adic regulator $R_{F,p}$ is a unit at~$p$. 

\end{abstract}
\maketitle

\tableofcontents

\newcommand{\twopartdef}[4]
{
	\left\{
		\begin{array}{ll}
			#1 & \mbox{if } #2 \\
			#3 & \mbox{if } #4
		\end{array}
	\right.
}
\newcommand{\I}[1]
{
\mathfrak{#1}
}

\newcommand{\D}[1]
{
#1^{\vee}
}
\newcommand{\U}[1]
{
	{#1}^{\times}
}

\newcommand{\CharFun}[1]
{
\textbf{1}_{#1}
}

\newcommand{\Al}[0]
{
\mathcal{O}
}
\newcommand{\Mod}[0]
{
\text{ mod }
}
\newcommand{\Minus}[0]
{
\backslash
}
\newcommand{\AC}[1]
{
	\overline{#1}
}
\newcommand{\MatTwo}[4]
{
\left(
\begin{array}{ccc}
#1 & #2 \\
#3 & #4 \\
\end{array}
\right)
}
\newcommand{\MatThree}[9]
{
\left(
\begin{array}{ccc}
#1 & #2 & #3 \\
#4 & #5 & #6 \\
#7 & #8 & #9 \\
\end{array}
\right)
}
\newcommand{\invlim}[1]
{
\lim_{\stackrel{\longleftarrow}{#1}}
}
\newcommand{\Sc}[1]
{
	\mathcal{#1}
	}
\newcommand{\IP}[1]
{
	\left\langle #1 \right\rangle
	}
\theoremstyle{definition}
\newtheorem{theorem}{Theorem}[section]
\newtheorem{lemma}[theorem]{Lemma}
\newtheorem{corollary}[theorem]{Corollary}
\newtheorem{proposition}[theorem]{Proposition}
\newtheorem{conjecture}[theorem]{Conjecture}
\newtheorem{definition}[theorem]{Definition}
\newtheorem{example}[theorem]{Example}
\newtheorem{question}[theorem]{Question}
\newtheorem{heuristic}[theorem]{Heuristic}
\theoremstyle{remark}
\newtheorem{remark}[theorem]{Remark}
\newtheorem*{claim}{Claim}
\newtheorem*{proofofclaim}{Proof of Claim}
\setlength{\parindent}{0cm}
\section{Introduction}
Let $F$ be a number field, and let $r_1$ and $r_2$ be the number of real and complex places of $F$, respectively, so that $[F:\Q] = r_1+2r_2$. The Leopoldt conjecture predicts that the number of $\Z_p$-extensions of $F$ is $r_2+1$. Put another way, it asserts that the Galois group of the maximal, abelian pro-$p$ extension of $F$ unramified outside the places above $p$ has $\Z_p$-rank equal to $r_2+1$. By considering the global Euler-Poincar\'{e} characteristic formula, this statement is equivalent to the vanishing $H^2(G_{F,S_p \cup S_{\infty}},\Q_p) = 0$, where $S_p$ and $S_{\infty}$ are the places of $F$ above $p$ and $\infty$, respectively, and $G_{F,S_p \cup S_{\infty}}$ is the Galois group of the maximal extension of $F$ unramified outside $p$ and $\infty$. \medskip

One can consider a mod $p$ analogue of this statement. The type of question we will be asking throughout this note is the following:
\begin{question}\label{main} Is $H^2(G_{F,S_p \cup S_{\infty}},\Z/p\Z) = 0$ for almost all primes $p$?
\end{question}
In this note, ``almost all'' refers to either ``all but finitely many'' or ``all outside a set of density zero.'' We would be satisfied with either answer, and in our heuristics, we clearly state the intended meaning. The question can be viewed as  asking if for almost  all $p$ the deformation of the trivial mod $p$ representation of $G_F$ is unobstructed.   \medskip

Note that the vanishing in Question \ref{main} implies, by the global Euler-Poincare characteristic formula, that $$\dim H^1(G_{F,S_p \cup S_{\infty}},\Z/p\Z) = \dim \Hom(G_{F,S_p \cup S_{\infty}},\Z/p\Z) =  r_2+1,$$ which implies  Leopoldt's conjecture in the stronger form that  the Galois group of the maximal, abelian  pro-$p$ extension of $F$ unramified outside the places above $p$  is $\Z_p^{r_2+1}$. Thus we think of the expected affirmative answer to our question as a  mod $p$ Leopoldt conjecture  for almost all primes $p$. \medskip

Question \ref{main} has an affirmative answer in the case  $F=\Q$ (since $H^2(G_{\Q,S_p \cup S_{\infty}},\Z/p\Z) = 0$ for all $p>2$, and for $p=2$ we have $\dim H^2(G_{\Q,S_2 \cup S_{\infty}},\Z/2\Z) = 1$), and when  $F$ is an imaginary quadratic field. However,  in any case where we have a unit of the number field of infinite order,   for instance $F$ real quadratic,  we do not know the answer to our  question, despite the fact that $H^2(G_{F,S_p \cup S_{\infty}},\Q_p) = 0$ is easy for $F$ real quadratic fields. In this case (see \S 4)  our question is a direct  analog of the classical question  if almost all primes $p$ are non-Wieferich, i.e., $2^{p-1}$ is not 1 mod $p^2$  for almost all $p$. \medskip

Question~\ref{main} and the heuristics for it have been studied before by Gras and others, see \cite{Gras} and the references therein.\footnote{We thank G. Gras  and C. Maire for pointing this out to us.} A field that satisfies $H^2(G_{F,S_p},\Z/p\Z) = 0$ is called $p$-rational; see \cite[p.~162]{Movahhedi-Nguyen} (note that $H^2(G_{F,S_p},\Z/p\Z) = H^2(G_{F,S_p \cup S_{\infty}},\Z/p\Z)$ for $p>2$). The expected abundance of $p$-rational fields has been formulated in various places. For recent applications and conjectures see \cite{Greenberg} or \cite{Hajir-Maire}.
 \medskip

More generally in this paper, we will be asking questions related to the vanishing of degree two Galois cohomology groups with coefficients arising adjoint representations thar arise   from compatible systems of Galois representations.   In other words, we ask if for most primes $p$ the deformation theory of mod $p$ representations arising from a fixed compatible system is unobstructed. The questions seem extremely hard to answer,  but  we record    some computational evidence in support of our guesses.  \medskip

We give a brief description of the contents. In \S 2 we state the dimension conjecture for deformation rings (due to Mazur) which motivated our work. In \S 3 we recall results of Weston which prove that for almost all primes $p$ the deformations of mod $p$ representations arising from  a fixed newform of weight $k \geq 2$  is unobstructed which motivated our mod $p$ Leopoldt conjecture. In \S 4 we discuss the mod $p$ Leopoldt conjecture and some (meagre) computational evidence for it.  (The reference \cite{Gras} gives many references to the rich literature about these questions.)  In \S 5 we consider the characteristic 0 situation briefly. Here the classical Leopoldt conjecture is open, and in a more automorphic setting, the smoothness of deformations of $p$-adic representations arising from  (for example) Bianchi forms is not known. In \S 6 we consider the conjectural  analog of Weston's result for  mod $p$ compatible systems arising from classical weight one forms. In \S 7 we consider the finer question if  the ordinary weight one  deformation ring for almost all primes $p$ of the mod $p$ representations $\{\rhobar_{f,p}\}$ arising from  a fixed classical weight one newform $f$ is smooth (which implies that the ordinary deformation ring is smooth, and that $H^2(S \cup \{p,\infty\},\Ad(\rhobar_{f,p}))$=0 with $S$ consisting of primes that divide the level of $f$). This can be expressed qualitatively by saying that mod $p$ representations arising from $p$-adic weight one ordinary forms of fixed tamel level $N$ are typically ramified at $p$. In \S 8 we consider the analog of this question for ramification away from $p$. For example we ask:  Fix primes $p \neq \ell$, then is it true that for all but finitely many primes  $q$ there is a (mod $p$) Hecke eigenform $f$ in $S_2(\Gamma_0(q\ell),\F_p)$ such that  the corresponding mod $p$ representation $\bar\rho_f: G_\Q \ra \GL_2(k)$, with $k$ a finite field of characteristic $p$,  is  irreducible and ramified at $q$  and  $\ell$? The question is challenging when there is a form in $S_2(\Gamma_0(\ell),\F_p)$ that gives rise to an irreducible mod $p$ representation for which $q$ is a level raising prime. We can't rule out the possibility that all the modular forms in $S_2(\Gamma_0(q\ell),\F_p)$  give rise to representations that already arise from $S_2(\Gamma_0(\ell),\F_p)$. \medskip

The  heuristic we employ in this paper  is similar to that used in \cite{Gras}  (see also its references),  \cite[Prop.~7.2]{JKPSZ} and to some Cohen-Lenstra style heuristics. We chose the reference to Wieferich primes in the title of the article, since their expected occurrence is the most basic instance where such a heuristic seems to have been applied first. \medskip

  We would like to thank  Frank Calegari and  Jo\"el Bella\"iche  for helpful correspondence about the questions explored in this note. The last author  wondered about a ``mod $p$ Leopoldt'' conjecture and the natural heuristic for it  in 2014, and subsequently corresponded with Frank Calegari about it:  \S  \ref{sec:Artin}  and \S  \ref{ExampleCrefined}  are due to Calegari and date from that correspondence.     We would also like to thank John Coates, Ravi Ramakrishna, Romyar Sharifi and Jack Thorne for useful correspondence and conversations. \medskip

\section{Review of Deformation Theory and Dimension Conjectures}
In this section, we review the basics of Galois deformations and recall a dimension conjecture of Mazur. Let $F$ be a number field. We will often restrict ourselves to the case where $F$ is either totally real or CM (the questions  we ask seem  hard  even over quadratic fields, and sometimes for $\Q$ itself), but for now assume $F$ is arbitrary. Let $n \geq 1$ be an integer, and let $p$ be a prime such that $(p,n) = 1$. Let $S \supset S_p \cup S_{\infty}$ be a finite set of places of $F$. We will let $G_{F,S}$ denote the Galois group of the maximal algebraic extension of $F$ unramified outside $S$. Suppose that $\bar{\rho}: G_{F,S} \to \GL_n(k)$ is an irreducible representation, where $k$ is a finite field of characteristic $p$. If $F$ is totally real or CM, $\bar\rho$ will typically be odd.\medskip
 
Let $W(k)$ denote the Witt vectors of $k$, and let $\CNL_{W(k)}$ be the category of complete Noetherian local $W(k)$-algebras with residue field $k$. We will denote by $\Sc{F}: \CNL_{W(k)} \to \text{Sets}$ the functor which takes an object $A \in \CNL_{W(k)}$ to the set of deformations $\rho: G_{F,S} \to \GL_n(A)$ of $\bar{\rho}$. Mazur \cite{Maz} proved that $\Sc{F}$ is a representable functor, and we denote the representing object by $R^{\univ}$, known as the universal deformation ring. We have the following presentation result due to Mazur.
\begin{proposition} Let $h^i = \dim_k H^i(G_{F,S},\Ad(\bar{\rho}))$. Then there is a presentation 
$$R^{\univ} \cong W(k) \llbracket X_1,\ldots,X_{h^1} \rrbracket / (f_1,\ldots,f_{h^2})$$ inducing an isomorphism of mod $p$ Zariski tangent spaces. 
\end{proposition}
Recall, also, the Euler-Poincar\'{e} formula:
\begin{proposition} If $h^i$ is as in the previous proposition, then 
$$h^1-h^2 = (1+r_2) + (n^2-1)[F:\Q] - \sum_{v|\infty} \dim_k H^0(G_v,\Ad^0(\bar{\rho})).$$
\end{proposition}
These two propositions lead to the following conjecture:
\begin{conjecture}[{\cite[Sec.~4]{Gouvea}}] Suppose that $S$ is finite and contains $S_p\cup S_\infty$. Then $R$ is flat over $W(k)$ and the relative dimension of $R$ over $W(k)$ is given by
$$\delta(\bar{\rho}) = (1+r_2)+(n^2-1)[F:\Q]-\sum_{v| \infty} \dim_k H^0(G_v,\Ad^0(\bar{\rho})).$$
\end{conjecture}
The conjecture seems to go back to Mazur, who observes in \cite[1.10, Lem.~4]{Maz} that if $n=1$, then it is the same as the classical Leopoldt conjecture as stated in the introduction. Thus, Mazur's conjecture is a higher dimensional analogue of this classical question. Let us note that our reference is the {\em Dimension Conjecture} stated by Gouv\^{e}a in \cite[Sec.~4]{Gouvea}, who suggests Mazur as the source. A similar conjecture was stated by Flach. The statements made by Gouv\^{e}a and by Flach have in fact fewer hypotheses than the conjecture we state above, and in both cases counterexamples are known; see \cite{Sprang} and \cite{BlCh}.\medskip

One case where Mazur's conjecture is obvious is the case when $h^2 = 0$, i.e., when we have
$H^2(G_{F,S},\Ad(\bar{\rho})) = 0$. In this case, the lifting problem for $\bar{\rho}$ is unobstructed. We will be examining the case when this second cohomology group vanishes for generic characteristic when examining residual representations which arise from some compatible systems. We make the following definition.
\begin{definition} Let $(\bar{\rho}_{\lambda})_{\lambda}$ denote an $L$-rational compatible system of mod $p$ Galois representations, where $\lambda$ runs over the finite places of $L$ and the $\bar{\rho}_{\lambda}: G_{F,S \cup S_{l(\lambda)}} \to GL_n(k_{\lambda})$ are Galois representations (here $l(\lambda)$ is the residue characteristic of~$\lambda$), in the sense of \cite{Khare}. Then the system is said to be \emph{generically unobstructed} if $H^2(G_{F,S \cup S_{l(\lambda)}},\Ad(\bar{\rho}_{\lambda})) = 0$ for almost all $\lambda$. 
\end{definition}

\section{Generic Unobstructedness in Modular Forms Setting}
In this section, we recall a result of Weston \cite{Wes} which motivates the idea of this note. Let $f$ denote a newform of level $N$, weight $k \geq 2$, and character $\epsilon$, and let $K = \Q(a_n(f))$ denote the number field gotten by adjoining the Fourier coefficients of the $q$-expansion of $f$. Let $S$ be a finite set of places of $\Q$ which includes all primes dividing $N$ and the infinite place. Consider the compatible system of representations $(\rho_{f,\lambda})_{\lambda}$, where $\lambda$ runs over the finite places of $K$. We denote by $\bar{\rho}_{f,\lambda}:G_{\Q,S \cup \{l\}} \to \GL_2(k_{\lambda})$ the reduction of $\rho_{f,\lambda}$, which for almost all $\lambda$ is an irreducible representation. Here we simply write $l$ for $l(\lambda)$. \medskip

Let $R_{f,S,\lambda}$ be the universal deformation ring parametrizing deformations of $\bar{\rho}_{f,\lambda}$ which are unramified outside $S \cup \{l \}$. Then Weston proves the following:
\begin{theorem} \label{thm:weston}
\begin{itemize}
\item[(1)] If $k \geq 3$, then for almost all places $\lambda$, the deformation ring is unobstructed, i.e. $H^2(G_{\Q,S \cup \{l \}},\Ad(\bar{\rho}_{f,\lambda})) = 0$ and $R_{f,S,\lambda} \cong W(k_{\lambda}) \llbracket X_1,X_2,X_3 \rrbracket$. 
\item[(2)] If $k=2$, then the above is true for all $\lambda$ outside a set of places of density zero. More precisely, if $T$ is the set of all $\lambda$ such that $a_l(f)^2 \equiv \epsilon(l) \Mod \lambda$ and $\bar\rho_\lambda$ restricted to $G_l$ is semisimple, then unobstructedness can only fail for finitely many $\lambda$ outside~$T$.
\end{itemize}
\end{theorem}

\begin{remark}
The case of weight $k=1$ is subtler, and relates to one of the main themes of the paper as we see in   \S  \ref{sec:Artin}  and \S  \ref{ExampleCrefined}.
\end{remark}

\begin{proof}[Sketch of Proof] For the full proof of the theorem, see \cite[Sect.~5.3]{Wes}. The sketch is as follows. Consider the minimal deformation ring $R$ associated to $\bar{\rho}_{f,\lambda}$, where $\lambda$ is of characteristic $l \gg 0$. By the modularity lifting theorems of Wiles and Taylor-Wiles, one shows that $R$ is isomorphic to a Hecke ring acting on $S_2(\Gamma_1(N),\Sc{O})_{\I{m}_{\lambda}}$. This Hecke ring is isomorphic to $W(k_{\lambda})$, whence the isomorphism $R \cong W(k_{\lambda})$ follows. Thus, the Zariski tangent space of $R$ is trivial, which means that a Selmer group $H^1_{\Sc{L}}(G_{\Q,S \cup \{l \}},\Ad^0(\bar{\rho}_{f,\lambda})) = 0$. One then applies the Greenberg-Wiles formula to obtain the vanishing of the corresponding dual Selmer group $H^1_{\Sc{L}^{\perp}}(G_{\Q,S \cup \{l\}},\Ad^0(\bar{\rho}_{f,\lambda})(1)) = 0$.\medskip

By examining the Poitou-Tate exact sequence, we obtain
$$0 \to H^1(G_{\Q,S \cup \{l\}},\Ad^0(\bar{\rho}_{f,\lambda})) \to \bigoplus_{v \in S \cup \{l\}} H^1(\Q_v,\Ad^0(\bar{\rho}_{f,\lambda}))/\Sc{L}_v \to $$ $$ 0 \to H^2(G_{\Q,S \cup \{l\}},\Ad^0(\bar{\rho}_{f,\lambda})) \to \bigoplus_{v \in S \cup \{l\}} H^2(\Q_v,\Ad^0(\bar{\rho}_{f,\lambda})) \to H^0(G_{\Q},\Ad^0(\bar{\rho}_{f,\lambda})(1)).$$
If $k > 2$, then the local $H^2$-terms vanish for sufficiently large $l$. If $k=2$, then the local term $H^2(\Q_l,\Ad^0(\bar{\rho}_{f,\lambda}))$ may not vanish if $l$ is in the exceptional set of primes described in the theorem. This is why the two cases are split in the statement of the theorem. 
\end{proof}
Observe that the theorem is saying that the compatible system of
residual representations attached to newforms is generically unobstructed as defined in the previous section. For regular algebraic
irreducible polarized compatible system of representations in
\cite{BLGGT}, the method of Weston should generalize to prove generic
unobstructedness at a positive density of primes as $R = \T$ theorems
are available. There is work by David Guiraud on this question; see \cite{Guiraud}. A key question that requires thorough analysis is whether the local $H^2$-terms vanish using local-global compatibility results for Galois representations attached to regular, algebraic, conjugate, self-dual, cuspidal (RACSDC) forms provided by Ana Caraiani \cite{Car}. \medskip
\begin{remark}
Let us also note that a simple statistical model might suggest that the set $T$ in the statement of Theorem~\ref{thm:weston} is finite: As discussed in \cite[p.~157]{Mazur-Fern} one expects (if $f$ is $\Q$-rational) for infinitely many primes $l$ that $a_l(f)^2 = \epsilon(l)$ but that for any $x\gg0$ the number of such $l$ is of the order $O(\log\log x)$. A statistical model can be based on the Sato-Tate conjecture, now a theorem of Clozel, Harris and Taylor; see \cite{CHT} and \cite{Tay}; chances of $a_l(f)^2=\epsilon(l)$ are of the order $O(\frac1{\sqrt{l}})$. However there is also the issue of the semisimplicity of $\bar\rho_\lambda$ restricted to $G_l$. This might be expected to occur statistically $\frac1l$ many times, and this should be independent of $a_l(f)^2=\epsilon(l)$. Now since $\sum_l\frac1{l\sqrt{l}}$ is finite, one should expect $T$ to be finite. Results and conjectures in this direction are formulated in \cite{DW} and \cite{Gamzon}.
\end{remark}

After we leave the setting of RACSDC forms, we encounter  (as far as we know) open questions such as the following:
\begin{question} Let $K/\Q$ be imaginary quadratic and let $\pi$ be a regular, algebraic cusp form on $\GL_2(\A_K)$. Consider the compatible system $(\rho_{\pi,\lambda})$ associated to $\pi$. Is it true that for almost all $\lambda$, we have an injection $H^2(K_{S \cup \{l \}},\Ad(\bar{\rho}_{\pi,\lambda})) \hookrightarrow \bigoplus_{v \in S \cup \{l\}} H^2(K_v,\Ad(\bar{\rho}_{\pi,\lambda}))$?
\end{question}
 The difficulty of answering this question is  the following. Automorphy lifting theorems as in \cite{CG} would prove under some hypotheses  that a certain minimal Selmer group $H^1_{\Sc{L}}(G_{K,S \cup \{l\}},\Ad^0(\bar{\rho}_{\pi,\lambda})) = 0$ for almost all $\lambda$. However, the Greenberg-Wiles formula does not provide the vanishing of the corresponding dual Selmer group because we are not in the \emph{balanced} situation that Weston operated in for his theorem. In fact, in this ``defect one'' setting, the dual Selmer group has dimension $\dim H^1_{\Sc{L}^{\perp}}(G_{K,S \cup \{l\}},\Ad^0(\bar{\rho}_{\pi,\lambda})(1)) = 1+ \dim  H^1_{\Sc{L}}(G_{K,S \cup \{l\}},\Ad^0(\bar{\rho}_{\pi,\lambda}))=1$. Thus for all but finitely many $\lambda$, 
the methods in \cite{CG} would prove    
$$\dim H^2(G_{K,S \cup \{l\}},\Ad(\bar{\rho}_{\pi,\lambda})) \leq 1+\sum_{v \in S \cup \{l\}} \dim H^2(K_v,\Ad(\bar{\rho}_{\pi,\lambda})),$$ and not the sharper inequality $$\dim H^2(G_{K,S \cup \{l\}},\Ad(\bar{\rho}_{\pi,\lambda})) \leq \sum_{v \in S \cup \{l\}} \dim H^2(K_v,\Ad(\bar{\rho}_{\pi,\lambda})),$$  which we might expect to hold for almost all $\lambda$. 

\section{The Trivial Motive}
In the introduction, we posed the following question which has been studied in \cite{Gras}: If $F$ is a number field and $G_{F,p}$ denotes the maximal extension of $F$ unramified outside $p$ and $\infty$, then is $H^2(G_{F,p},\Z/p\Z) = 0$ for all but finitely many $p$? As remarked after the question, this is easy in the case in the case of $F = \Q$ and in the case when $F$ is an imaginary quadratic field. In this section, we will examine mainly  the real quadratic setting. \medskip

The heuristic we give in the real quadratic situation, which is close to \cite{Gras},  is very similar to counting Wieferich primes.   We refer the reader to \cite{Gras} for much finer heuristics about this question.

\begin{definition} A prime $p$ is a \emph{Wieferich prime} if $2^{p-1} \equiv 1 \Mod p^2$. 
\end{definition}
Notice that $2^{p-1} \equiv 1 \Mod p$ is Fermat's Little Theorem, so Wieferich primes require divisibility by an additional power of $p$. It is not currently known whether there are infinitely many Wieferich primes, nor whether there are infinitely many non-Wieferich primes. These are incredibly difficult to find. However, there are guesses as to how many there should be, based on the following heuristic argument from \cite{CDP}. \medskip

Since $2^{p-1} \equiv 1 \Mod p$, we know that mod $p^2$, we must have $2^{p-1} \equiv 1+kp \Mod p^2$, where $0 \leq k \leq p-1$. If $k=0$, then $p$ is a Wieferich prime. The probability of this happening should be $\frac{1}{p}$. If we treat each prime as ``independent'' events, then the number of Wieferich primes less than or equal to some number $X$ should be
$$\sum_{p \leq X} \frac{1}{p}.$$ This sum grows like $\log \log X$. \medskip

A nice way of viewing this analysis is the following. We have a set, namely \[K = \ker \left(\Z/p^2\Z \to \Z/p\Z \right)\] and we have a given point that we want to hit, namely $1 + p^2 \Z \in K$. We are treating our object $2^{p-1} + p^2\Z \in K$ as a random point in $K$ and asking for the probability that this random point is our desired target point.\medskip

With this in mind, we return to our original question about the vanishing of $H^2(G_{F,p},\Z/p\Z)$.  By the Euler-Poincar\'{e} characteristic formula, if $h^i = \dim_{\Z/p\Z} H^i(G_{F,p},\Z/p\Z)$, we know
$$h^1 = 1+h^2.$$ If $p$ does not divide the class number of $F$, the group $H^1(G_{F,p},\Z/p\Z) = \Hom(G_{F,p},\Z/p\Z)$ is dual to the $p$-part of the ray class group of $F$ of conductor $p^2$. We therefore have an exact sequence,
$$1 \to \frac{\Sc{O}_F^{\times} \cap (1+p\Sc{O}_F)}{\Sc{O}_F^{\times} \cap (1+p^2\Sc{O}_F)} \to \frac{(1+p\Sc{O}_F)}{(1+p^2\Sc{O}_F)} \to H^1(G_{F,p},\Z/p\Z)^{\vee} \to 1.$$ 
By counting dimensions, we see that $$h^1 = 2-\dim_{\Z/p\Z} \frac{\Sc{O}_F^{\times} \cap (1+p\Sc{O}_F)}{\Sc{O}_F^{\times} \cap (1+p^2\Sc{O}_F)}.$$ Comparing the two expressions, we see that $h^2 = 0$ precisely when $$ \dim_{\Z/p\Z}\frac{\Sc{O}_F^{\times} \cap (1+p\Sc{O}_F)}{\Sc{O}_F^{\times} \cap (1+p^2\Sc{O}_F)} = 1.$$ If $\epsilon$ denotes the fundamental unit of $F$, then $h^2 \neq 0$  is the same as saying that $\epsilon^{p^2-1} \in \Sc{O}_F^{\times} \cap (1+p\Sc{O}_F)$ is $p$-th power. Equivalently, that $\epsilon^{p^2-1} \equiv 1 \Mod p^2\Sc{O}_F$. One would expect this to happen with probability $1/p$, and so the number of primes $p$ up to $X$ for which $h^2 \neq 0$ should be 
$$\sum_{p \leq X} \frac{1}{p},$$ which, as in the Wieferich primes setting, again grows like $\log \log X$. From this, we should expect a density one set of primes for which $H^2(G_{F,p},\Z/p\Z) = 0$. This seems very hard  to prove, but the work of \cite{Sil} suggests that assuming the $abc$ conjecture, one could show that this happens for at least  $c \log X$ primes $p \leq X$, where $c$ is a nonzero constant. We used magma to check the primes $1 < p<10000$ for which $H^2(G_{F,p},\Z/p\Z) \neq 0$ as $F$ ranges over real quadratic fields $F = \Q(\sqrt{D})$ for $2 \leq D \leq 30$. 
\begin{table}[htb]\begin{center}
\begin{tabular}{c|c}
$D$ & $p$ \\
\hline
$2$ & $13$, $31$ \\
$3$ & $103$ \\
$5$ & \\
$6$ & $7$, $523$ \\
$7$ & \\
$10$ & $191$, $643$ \\
$11$ & \\
$13$ & $241$ \\
$14$ & $2$ \\
$15$ & $181$, $1039$, $2917$ \\
$17$ & \\
$19$ & $79$ \\
$21$ &  \\
$22$ & $43$, $73$, $409$ \\
$23$ & $7$, $733$ \\
$26 $& $2683$, $3967$ \\
$29$ & $3$, $11$ \\
\end{tabular}\medskip
 \end{center}
\caption{}\label{TableZero}
\end{table}

The data is shown in the Table~\ref{TableZero} on page~\pageref{TableZero}. The entries indicate that in the small cases considered, the nonvanishing of $H^2(G_{F,p},\Z/p\Z)$ seems to be quite rare, but as  $\log \log X$ grows so slowly almost no amount data would be  convincing enough. Observe that the analysis above mirrors the earlier analysis of the Weiferich primes. Cf.~\cite{Katz} for similar results for Wieferich primes and  geometric analogs. In this paper we are exploring the Wieferich type phenomenon for Galois cohomology.\medskip

We can offer another perspective on the question, where now we let $F$ be any totally real number field. By a formula of Colmez, see \cite{Colmez}, if $\zeta_{F,p}(s)$ denotes the Deligne-Ribet $p$-adic zeta function of $F$, then 
$$\text{Res}_{s=1}\zeta_{F,p}(s) = \frac{2^dR_{F,p}h_F}{2\sqrt{D_F}},$$ where $R_{F,p}$ is the $p$-adic regulator of $F$. This is in parallel with the Dirichlet-Dedekind class number formula for the zeta function $\zeta_F(s)$:
$$\text{Res}_{s=1} \zeta_F(s) = \frac{2^dR_{F,\infty}h_F}{2\sqrt{D_F}},$$ where $R_{F,\infty}$ is the classical regulator of $F$. The existence of a pole at $s=1$ of $\zeta_{F,p}(s)$ is equivalent to Leopoldt's conjecture for $F$ and $p$. On the other hand, for almost all primes $p$, the vanishing of $H^2(G_{F,p},\Z/p\Z)$ is equivalent to $\text{Res}_{s=1}\zeta_{F,p}(s)$ being a $p$-adic unit. Thus, for almost all $p$, our question is a refinement of the Leopoldt conjecture that not only does the Deligne-Ribet $p$-adic zeta function have a pole at $s=1$, but that the residue is also a $p$-adic unit. 

\section{Characteristic Zero}
We switch gears and consider the characteristic zero question. Let $F$ be a number field, and let $S$ be a finite set of places of $F$ that contains all places above $p$ and $\infty$. Let $V$ be a $p$-adic representation of $G_{F,S}$ that is pure of weight $w \neq -1$. An example of such a $V$ would be the adjoint motive of a pure motive, which would be of weight zero. We have the following conjecture due to Jannsen.
\begin{conjecture}[{ \cite[Conj.~1]{Jan}}] We have
$$\dim H^2(G_{F,S},V) = -\dim H^0(G_{F,S}, V^*(1)) + \sum_{v \in S^{\infty}} \dim H^0(G_{F_v},V^*(1)).$$
\end{conjecture}
\begin{remark} By the Euler-Poincar\'{e} characteristic formula, this conjecture is equivalent to the statement 
$$\dim H^1(G_{F,S},V) = [F:\Q]\dim(V) + \sum_{v \in S^{\infty}} \dim H^0(G_{F_v},V^*(1)) - \sum_{v | \infty} \dim H^0(G_{F_v},V)$$ $$ + (\dim H^0(G_{F,S},V) - \dim H^0(G_{F,S},V^*(1))).$$
\end{remark}
In the setting $F=\Q$ and $V$ is the adjoint $p$-adic representation arising from a newform $f$ of weight $k \geq 2$, this is a known result due to Flach, Mazur, Weston, Diamond-Flach-Guo, Kisin et. al., e.g. \cite{DFG,Wes}. For weight $k=1$, this is proved in \cite{BD}  using  Baker-Brumer result on independence over $\overline{\Q}$ of $p$-adic logarithms of algebraic numbers that are independent over $\Q$. \medskip

If $V = \Q_p$, then this conjecture is equivalent to the Leopoldt conjecture for $F$ and $p$, which, as remarked in the introduction, is equivalent to $H^2(G_{F,p},\Q_p) = 0$. \medskip

We could study the case of Bianchi modular forms.  Let $K$ be an imaginary quadratic field, and let $\pi$ be a cuspidal cohomological form on $\GL_2(\A_K)$. This gives rise to a compatible system of representations of $G_K$, which we denote $(\rho_{\pi,\lambda})$. The representation $\Ad^0(\rho_{\pi,\lambda})$ is expected to be pure of weight zero since $\rho_{\pi,\lambda}$ is expected  to be pure. We can, therefore, try to examine Jannsen's conjecture in the case when $V = \Ad^0(\rho_{\pi,\lambda})$.   Using automorphy lifting methods,  for  forms that are ``generic'' it  can be proven that  $H^2(G_{K,S}, \Ad^0(\rho_{\pi,\lambda}))$ has dimension at most one, but it seems a hard problem to prove that the dimension is  0 as predicted by Jannsen.

\section{Artin Representations}\label{sec:Artin}
In this section, we will consider a system of representations arising from Artin representations over number fields. 
{This section is  due to Frank Calegari.}

First, we will work over $\Q$ and then move to the general situation of an arbitrary number field using Shapiro's lemma. Let us note again, that the heuristics we employ should be compared with that of Section~7 in the recent preprint \cite{JKPSZ}.
\subsection{Artin Representations over $\Q$} \label{subsec:artin-reps-over-Q} Let $W_{\C}$ be a finite dimensional $\C$-vector space, and let $\rho: G_{\Q} \to \Aut(W_{\C})$ be a continuous, irreducible representation. Let $E = \overline{\Q}^{\ker \rho}$ be the splitting field of the representation, so that $G = \Gal(E/\Q)$ is isomorphic to $\im \rho$. Denote by $S$ the set of places at which $E/\Q$ is ramified together with the place $\infty$.\footnote{In fact, in this subsection we can take for $E$ any finite Galois extension of $\bbQ$ such that $\rho$ is trivial on $G_E$; we shall use this observation in Subsection~\ref{subsec:ArtinOverNF}.} There is a canonical, minimal abelian field $K$ containing the traces of the action of $G$ on $W_{\C}$, and this field comes with an embedding $K \hookrightarrow \C$. For a place $l$ denote by $G_l\subset G$ the decomposition group at $l$. We will consider primes $p$ for which:
\begin{itemize}
\item[(Hyp 1)] The prime $p$ does not divide $|G|$.
\item[(Hyp 2)] The prime $p$ is unramified in $E$.
\item[(Hyp 3)] The prime $p$ does not divide the class number of $E$.
\item[(Hyp 4)] The prime $p$ is odd. {
\item[(Hyp 5)] For all $l$ in $S$, the restriction $W|_{G_l}$ contains no factor of dimension $1$ on which $\Frob_l$ acts as multiplication by $l$.
}
\end{itemize}
{ Note that (Hyp 5) holds if $p$ does not divide $l^m-1$ for any proper divisor $m$ of $[E:\Q]$, and hence for instance if $p\ge \max\{l^{[E:\Q]}\mid l \in S\}$.}
\medskip

Because of (Hyp 1), the representation $W_{\C}$ admits a model $W_{\Sc{O}}$ over the completion $\Sc{O}:=\Sc{O}_{K,\I{p}}$ at any prime $\I{p}$ above $p$. We will let $\overline{W}$ denote the mod $p$ reduction, namely $\overline{W} := W_{\Sc{O}}/\I{p}$. This reduction will be an absolutely irreducible, faithful representation of $G$ over $k:=\Sc{O}/\I{p}$. We could view $\overline{W}$ as a vector space over $\F_p$ instead of as one over $k$. When considering $\overline{W}$ as such a space, we will write it as $\overline{W}_{\F_p}$. Notice that $\overline{W}_{\F_p}$ is an irreducible $\F_p[G]$-module, and that $$\dim_{\F_p} \overline{W}_{\F_p} = [k:\F_p]\dim_{\C}(W_{\C}).$$ Note, also, that $\overline{W}_{\F_p} \otimes_{\F_p} k$ will decompose as the direct sum of $[k:\F_p]$ irreducible, non-isomorphic $k[G]$-modules, each of dimension $\dim_{\C}(W_{\C})$ over $k$, one of which is $\overline{W}$. Indeed, the irreducible $k[G]$-submodules of $\overline{W}_{\F_p} \otimes_{\F_p} k$ are indexed by the elements of $\Gal(k/\F_p)$ (see [Wiese]). In particular, if $M$ is an $\F_p[G]$-module, then we have an isomorphism
\begin{equation} \Hom_{\F_p[G]}(\overline{W}_{\F_p},M) \cong \Hom_{k[G]}(\overline{W},M \otimes_{\F_p} k) \label{eq:1} \end{equation}
{ For various explicit computations we also recall from \cite[Prop.~3.2]{Boston} the isomorphisms
\begin{equation}\label{eq:FromBoston}
\cO_E^\times/(\cO_E^\times)^p\oplus \F_p \cong \Ind_{G_\infty}^G\F_p \quad\hbox{and}\quad \cO_E/p\cO_E\cong \F_p[G]
\end{equation}
as $\F_p[G]$-modules, where $G_\infty$ is the decomposition group of $G$ at the infinite place; here we use (Hyp 2), so that $E$ contains no primitive $p$-th root of unity and is unramified at~$p$.} 
%

\medskip

The first question to tackle is to try and understand the expected dimension over $k$ of $H^1(G_{\Q,S\cup\{p\}},\overline{W})$, where as in the introduction, $G_{\Q,S\cup\{p\}}$ denotes the Galois group of the maximal extension of $\Q$ unramified outside $S\cup\{p\}$. Letting $h^i$ denote the dimension over $k$ of the corresponding cohomology group, the global Euler characteristic formula yields:
\begin{equation} h^1(G_{\Q,S\cup\{p\}},\overline{W})-h^2(G_{\Q,S\cup\{p\}},\overline{W}) = h^0(G_{\Q,S\cup\{p\}},\overline{W})+\dim(\overline{W}) - h^0(G_{\R},\overline{W}), \label{eq:2} \end{equation}
with $G_\R=\Gal(\C/\R)$ embedded into $G_\Q$. This gives a lower bound on $h^1(G_{\Q,S\cup\{p\}},\overline{W})$. \medskip

On the other hand, by (Hyp 3), we can understand $h^1(G_{\Q,S\cup\{p\}},\overline{W})$ in terms of global units. Indeed, the inflation-restriction sequence gives an exact sequence 
$$0 \to H^1(G,\overline{W}^{G_{E,S\cup\{p\}}}) \to H^1(G_{\Q,S\cup\{p\}},\overline{W}) \to H^1(G_{E,S\cup\{p\}},\overline{W})^G \to H^2(G,\overline{W}^{G_{E,S\cup\{p\}}});$$ 
here $G_{E,S\cup\{p\}}$ denotes the Galois group of the maximal extension of $E$ unramified outside the places of $E$ above $S\cup\{p\}$. Since the cardinality of $\overline{W}{}^{G_{E,S\cup\{p\}}} = \overline{W}$ is prime to $|G|$ by (Hyp 1), we get that $H^i(G,\overline{W}) = 0$ for $i>0$. { Moreover by (Hyp 5), the inclusion $H^1_{{\rm unr}}(G_{E_\lambda},\overline{W})^{G_l}\to H^1(G_{E_\lambda},\overline{W})^{G_l}$ is an isomorphism for $\lambda$ a place of $E$ above~$l$: the map is isomorphic to $\Hom_{G_l}(\Z,\overline{W})\to \Hom_{G_l}(E^\times_\lambda,\overline{W})$ induced from the valuation map $E_\lambda^\times\to\Z$, and so its cokernel is $\Hom_{G_l}(\cO_{E_\lambda}^\times,\overline{W})$ vanishes by (Hyp 5). It follows that  the inclusion $H^1(G_{E,\{p,\infty\}},\overline{W})^G\to H^1(G_{E,S\cup\{p\}},\overline{W})^G$ is an isomorphism.} Thus, there are isomorphisms
\begin{eqnarray*}
H^1(G_{\Q,S\cup\{p\}},\overline{W}) &\cong &H^1(G_{E,S\cup\{p\}},\overline{W})^G \\&\cong &H^1(G_{E,\{p,\infty\}},\overline{W})^G \\ &\cong & (\Hom(G_{E,\{p,\infty\}},\F_p) \otimes_{\F_p} \overline{W})^G.
\end{eqnarray*}
Now by (Hyp 3), we may identify $\Hom(G_{E,\{p,\infty\}},\F_p)$ with the dual of the $p$-part of the ray class group of conductor $p^2$. That is, { we have a right exact sequence of $\F_p[G]$-modules
\begin{equation}\label{SES-WtOne}
\cO_E^\times\otimes_\Z\F_p \to (\cO_{E_p}^\times)/(1+p^2\cO_{E_p})\otimes_\Z\F_p \to H^1(G_{E,\{p,\infty\}},\F_p)^{\vee} \to 0,
\end{equation}
where $\cO_{E_p}$ is the completion of $\cO_E$ at $p$. The central term is isomorphic to the group $(1+p\cO_{E_p})/(1+p^2\cO_{E_p})$, and via the $p$-adic logarithm map $1+pz \Mod p^2 \mapsto z \Mod p$ the latter is isomorphic to $\cO_E/p\cO_E\cong \prod_{v|p} k_v $ as an $\F_p[G]$-module, with $k_v$ the residue field of $E$ at $v$. Denote by $s$ the exponent of the finite group $\prod_{v|p}k_v^\times$. then $H^1(G_{E,\{p,\infty\}},\F_p)^{\vee}$ is isomorphic to the cokernel of the $G$-equivariant homomorphism
\begin{equation}\label{SES-WtOne2}
\cO_E^\times \to \cO_E/p\cO_E,\alpha\mapsto \frac1p\big(\alpha^s-1\big).
\end{equation}
}
We have the following lemma:
\begin{lemma} There are equalities
\begin{eqnarray*}
\dim_k \Hom_{k[G]}(\overline{W},(\Sc{O}_E^{\times}/(\Sc{O}_E^{\times})^p)\otimes_{\F_p} k) &=& \dim_{\F_p} \Hom_{\F_p[G]}(\overline{W}_{\F_p},\Sc{O}_E^{\times}/(\Sc{O}_E^{\times})^p) \\&=& h^0(G_{\R},\overline{W}) - h^0(G_{\Q,S\cup\{p\}},\overline{W}).
\end{eqnarray*}
\end{lemma}
\begin{proof}
It suffices to show the equality between the first and the third term. Using (\ref{eq:FromBoston}), we deduce that
\begin{eqnarray*}
\lefteqn{ \Hom_{k[G]}(\overline{W},\Sc{O}_E^{\times}/(\Sc{O}_E^{\times})^p\otimes_{\F_p}k) \oplus  \Hom_{k[G]}(\overline{W},k)}\\
&\cong &  \Hom_{k[G]}(\overline{W},\Ind_{G_\infty}^Gk) \cong  \Hom_{k[G_\infty]}(\overline{W},k) .
\end{eqnarray*}
Taking dimensions, the result follows.
\end{proof}
Let $\delta(W):=h^0(G_{\R},\overline{W}) - h^0(G_{\Q,S\cup\{p\}},\overline{W})$. Tensoring the right exact sequence resulting from (\ref{SES-WtOne2}) with $\overline{W}_{\F_p}$, and taking $G$-invariants, which under (Hyp1) is an exact operation, produces a sequence of $\F_p$-vector spaces of the form
\begin{equation}\label{eq:random-map}
\F_p^{\delta(W)} \stackrel{\alpha_p}\to \F_p^{\dim (\overline{W})} \to \F_p^{h^1(G_{\Q,S\cup\{p\}},\overline{W})} \to 0.
\end{equation}
From equation (\ref{eq:2}), we find that $h^2(G_{\Q,S\cup\{p\}},\overline{W})=\dim\ker\alpha_p$. Thus $h^2(G_{\Q,S\cup\{p\}},\overline{W})$ vanishes precisely when $\alpha_p$ is injective. We can, therefore, try to determine when this is zero by treating $\alpha_p$ as a ``random'' map between $\F_p$-vector spaces, and asking when this random map is injective. 

\begin{lemma} Let $n \leq m$ be integers of size at least $1$. The probability that a random map $\F_p^n \to \F_p^m$ is injective is 
$$\prod_{i=0}^{n-1} \left(1-\frac{1}{p^{m-i}}\right).$$
\end{lemma}
\begin{proof} We want the column space of a matrix representing this random map to be $n$-dimensional. First, we can count the number of $n$-dimensional subspaces of $\F_p^m$. There are 
$$(p^m-1)(p^m-p)\cdots (p^m-p^{n-1})$$ different bases of $\F_p^m$ of size $n$. We then divide by the number of possible bases for a given $n$ dimensional subspace in order to obtain the number of $n$ dimensional subspaces of $\F_p^m$. This number is
$$(p^n-1)(p^n-p)\cdots (p^n-p^{n-1}),$$ giving
$$\frac{(p^m-1)(p^m-p)\cdots (p^m-p^{n-1})}{(p^n-1)(p^n-p)\cdots (p^n-p^{n-1})}$$ total subspaces of $\F_p^m$ of dimension $n$. There are then $(p^n-1)(p^n-p) \cdots (p^n-p^{n-1})$ surjective linear maps from $\F_p^n$ to a given $n$-dimensional subspace of $\F_p^m$. Thus, there are
$$(p^m-1)(p^m-p)\cdots (p^m-p^{n-1})$$ injective linear maps $\F_p^n \to \F_p^m$. There are $p^{nm}$ total linear maps. Thus, the probability that a random map is injective is given by the desired formula.
\end{proof}
Recall the following definition.
\begin{definition} An Artin representation $W_{\C}$ is \emph{totally even} (resp. \emph{totally odd}) if complex conjugation $c \in G$ acts by $+1$ (resp. $-1$). 
\end{definition}
We also classify Artin representations as one of three types:
\begin{lemma} Let $W_{\C}$ be an irreducible Artin representation over $\Q$. Then $W$ is either:
\begin{itemize}
\item[(A)] Totally odd or trivial. In this case $\delta(W) = 0$.
\item[(B)] Totally even but non-trivial. In this case $\delta(W) = \dim(W)$.
\item[(C)] Not of type (A) or (B). In this case, $0<\delta(W) < \dim(W)$. 
\end{itemize}
\end{lemma}
\begin{proof} The lemma is an easy consequence of the definition of $\delta(W)$ coupled with the previous definition.
\end{proof}
We formulate a heuristic:
\begin{heuristic}\label{HeuristicAlpha}
The map $p\mapsto \alpha_p$ behaves like a random map of vector spaces.
\end{heuristic}
With the previous two lemmas and the preceding remarks, we deduce the following:
\begin{proposition} \label{prop:6.5}Let $W_{\C}$ be an irreducible Artin representation over $\Q$. 
\begin{itemize}
\item[(i)] If $W$ is of type (A), then $H^2(G_{\Q,S\cup\{p\}},\overline{W}) = 0$ for all but finitely many primes~$p$.
\item[(ii)] If $W$ is of type (B) and if Heuristic~\ref{HeuristicAlpha} holds, then there are approximately $\log \log X$ primes $p \leq X$ for which $H^2(G_{\Q,S\cup\{p\}},\overline{W}) \neq 0$.
\item[(iii)] If $W$ is of type (C) and if Heuristic~\ref{HeuristicAlpha} holds, then one has $H^2(G_{\Q,S\cup\{p\}},\overline{W}) = 0$ for all but finitely many primes~$p$.
\end{itemize}
\end{proposition}
\begin{proof} We need to consider the kernel of the map $\alpha_p\colon\F_p^{\delta(W)} \to \F_p^{\dim(W)}$. In case (i) when $W$ is of type (A), then as $\delta(W) = 0$, we find that $\alpha_p$ is always injective. Thus $H^2(G_{\Q,S\cup\{p\}},\overline{W}) = 0$ for all primes satisfying (Hyp 1)-(Hyp 5). 

In case (ii), when $W$ is of type (B), then $\delta(W) = \dim(W)$, and Lemma 6.2 tells us that the probability that  $\alpha_p$ is \emph{not} injective is $\frac{1}{p}$. Thus, we expect
$$\sum_{p \leq X} \frac{1}{p} \sim \log \log X$$ primes $p \leq X$ for which $H^2(G_{\Q,S\cup\{p\}},\overline{W}) \neq 0$.

Finally, in case (iii), when $W$ is of type (C), then $\delta(W) < \dim(W)$. Lemma 6.2 says that the probability that $\alpha_p$ is not injective is roughly $\frac{1}{p^{\dim(W)-\delta(W)+1}} \leq \frac{1}{p^2}$. Since the sum over all primes $$\sum_{p} \frac{1}{p^2}$$ converges, we expect only finitely many primes for which $H^2(G_{\Q,S\cup\{p\}},\overline{W}) \neq 0$. 
\end{proof}

\subsection{Artin Representations over Arbitrary Number Fields}\label{subsec:ArtinOverNF}
Now suppose $F/\Q$ is a number field, and consider an irreducible Artin representation $W_{\C}$ of $G_F$. Suppose the representation factors through $\Gal(E/F)$, and let $L$ denote the Galois closure over $\Q$ of $E$, and $S$ the set of places of $\Q$ that ramify in $L/\Q$ together with $\infty$ (we shall use $S$ also for the places in $F$, $E$ or $L$ above $S$). We will let $G = \Gal(L/\Q)$, $H = \Gal(L/F)$, $A = \Gal(L/E)$. For convenience, below is the corresponding diagram.
$$
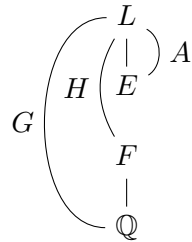
\begin{tikzpicture}[description/.style={fill=white,inner sep=2pt}, normal line/.style={->}]
\matrix (m) [matrix of math nodes, row sep=1em,
column sep=2.5em, text height=1.5ex, text depth=0.25ex]
{ L \\ E \\ F \\ \Q \\ };
\path[-] (m-1-1) edge (m-2-1);
\path[-] (m-1-1) edge[bend left = 60] node[right] {$A$} (m-2-1);
\path[-] (m-3-1) edge (m-4-1);
\path[-] (m-1-1) edge[bend right=90] node[left]{$G$} (m-4-1);
\path[-] (m-1-1) edge[bend right=30] node[left]{$H$} (m-3-1);
\end{tikzpicture} $$
We will choose a prime $p$ satisfying (Hyp 1) - (Hyp 5) for $L/\Q$ as in the previous section. By Shapiro's lemma, we have that
$$H^i(G_{F,S\cup\{p\}},M) = H^i(G_{\Q,S\cup\{p\}},\Ind_F^{\Q} M),$$ where $M$ is any $G_{F,S\cup\{p\}}$-module. We may write
$$\Ind_F^{\Q} W = \bigoplus U^{\mu_U},$$ where the various $U$ are irreducible representations of $G_{\Q}$. Note that the $U$'s will be representations of $G$. We make the following definition.
\begin{definition} 
\begin{itemize}
\item[(1)] We say that $W$ is of type (B) if at least one of the $U$ with $\mu_U > 0$ is of type (B). 
\item[(2)] We say that $W$ is of type (A) if every $U$ with $\mu_U>0$ is of type (A).
\item[(3)] Otherwise, $W$ is of type (C).
\end{itemize}
\end{definition}
Let $U$ be a representation with $\mu_U > 0$. By Frobenius reciprocity,
$$\Hom_H (W,U|_H) = \Hom_G(\Ind_F^{\Q} W,U) \neq 0.$$ Since $W$ is irreducible by hypothesis, this means that $W \subset U|_H$ for every $U$ with $\mu_U>0$. 
\begin{lemma} \label{lem:6.7}
\begin{itemize}
\item[(1)] $\dim_{\C}(W) > 1$. Then the representation $W$ is of type (B) if and only if $W$ is the restriction of a totally even irreducible representation of $G_{F^+}$, where $F^+ \subset F$ is totally real.
\item[(2)] If $\dim_{\C}(W) = 1$, then $W$ is of type (B) if and only if the maximal totally real subfield $F^+$ of $F$ is not $\Q$.
\end{itemize}
\end{lemma}
\begin{proof} First we prove (1). Suppose $W$ is the restriction of a totally even representation $U^+$ over a totally real field $F^+$. Then $\Ind_F^{\Q} W$ contains $\Ind_{F^+}^{\Q} U^+$, which is totally even and non-trivial unless $F^+ = \Q$ and $U^+ = \C$, which cannot happen since $\dim_{\C}(W) > 1$. \medskip

Conversely, suppose $W$ is of type (B). Then there exists a totally even representation $U$ of $G_{\Q}$ such that $W \subset U|_H$. Let $F^+$ denote the maximal totally real subfield of $F$. By Frobenius reciprocity, $W$ is also a subset of $U^+|_H$, where $U^+$ is an irreducible summand of $U$ restricted to $G_{F+}$. The claim is that $W = U^+|_H$. Equivalently, that $U^+|_H$ is irreducible. Since $U^+$ is totally even, the action of $G_{F^+}$ on $U^+$ factors through a finite extension $L^+/F^+$, where $L^+$ is totally real. On the other hand, $L^+$ and $F$ are totally disjoint over $F^+$, because the intersection is totally real and would therefore be contained in $F^+$ by the maximality of $F^+$. Hence, $\Gal(L^+F/F) = \Gal(L^+/F^+)$, and so the action of $G_{F}$ on the restriction $U^+|_H$ factors through the same image, and hence $U^+|_H$ is irreducible, as desired. \medskip

Statement (2) is easy. Indeed, if $W \cong \C$ and $F$ contains a totally real subfield $F^+$ which is not $\Q$, then $\Ind_F^{\Q} \C$ contains $\Ind_{F^+}^{\Q} \C$, which is nontrivial. Conversely, if $\Ind_F^{\Q} \C$ contains a nontrivial, totally even representation, the action of Galois factors through a nontrivial, totally real quotient $\Gal(F^+/\Q)$. 
\end{proof}
Lemma~\ref{lem:6.7} and Proposition~\ref{prop:6.5} lead to the following expectation:
\begin{itemize}
\item[(1)] If $W = \C$, then $H^2(G_{F,S\cup\{p\}},\overline{W}) \neq 0$ for infinitely many primes $p$ if and only if $F$ contains a totally real subfield $F^+ \neq \Q$.
\item[(2)] If $\dim_{\C}(W) > 1$, then $H^2(G_{F,S\cup\{p\}},\overline{W}) \neq 0$ for infinitely many primes $p$ if and only if $W$ extends to a totally even irreducible representation over a totally real subfield $F^+ \subset F$.
\end{itemize}

\subsection{An Example of type (C)} \label{ExampleC} Let $K/\Q$ be an imaginary cubic field with Galois closure $E/\Q$ and $\Gal(E/\Q) \cong S_3$. Let $$\rho: G_{\Q} \to \GL_2(\C)$$ be the corresponding weight $1$ representation, with underlying representation space $W_\C$ and with $\det(\rho) = \chi$ the corresponding quadratic character. Suppose $\sigma,\tau \in G_{\Q}$ generate $\Gal(E/\Q)$ when restricted to $E$. We will assume $\rho$ has the following form:
$$\rho(\sigma) = \MatTwo{\omega}{0}{0}{\omega^2},\quad \rho(\tau) = \MatTwo{0}{1}{1}{0},$$ where $\omega \in \C$ is a (primitive) cube root of unity. Let $p>3$ be a prime which is unramified in $E$ and coprime to the class number $h_E$. Then (Hyp 1) to (Hyp 4) are satisfied. { Depending on $E$ there will also be a small finite list of primes $p$ that need to be excluded for (Hyp 5) to hold. Since $G=S_3$, it is not hard to see that this latter list is contained in 
\begin{equation}\label{eq:DefH5}
H_5:=\{p\mid p\hbox{ is a prime number and $p$ divides $l^2-1$ for some }l\in S\}.
\end{equation}}

Denote by $\bar\rho$, $\bar\chi$ and $\overline W$ the reductions mod $p$ of $\rho$, $\chi$ and $W$ formed in the sense of Subsection~\ref{subsec:artin-reps-over-Q}. Using standard results on induction, and that $\rho$ is induced from a character of order $3$ of the fixed field of $\ker\chi$, one verifies that $\Ad^0(\bar\rho)\cong\overline W\oplus\bar\chi$. Concretely $\overline W$ is realized inside $\Ad^0(\bar{\rho})$ as submodule of matrices of the form $\MatTwo{0}{b}{c}{0}$ under the adjoint action of $G_{\Q,S\cup\{p\}}$. Using (\ref{eq:random-map}) one computes
$$\dim H^1(G_{\Q,S\cup\{p\}},\bar{\chi}) = 1,\quad \dim H^1(G_{\Q,S\cup\{p\}},\overline W) = 1 \text{ or } 2.$$ 
{In the case at hand $\delta(\overline W)=1$, $\dim \overline W=2$ and hence arguing as after equation (\ref{eq:random-map}), one has $h^2(G_{\Q,S\cup\{p\}},\overline W)=0$ if and only if $h^1(G_{\Q,S\cup\{p\}},\overline W)=1$. We have tested this numerically in some cases, under the additional hypothesis that $\bar{\rho}(\Frob_p)$ has order $3$. This latter condition is useful in Subsection~\ref{ExampleCrefined} where a refinement of the present analysis is given.

\medskip

We now explain how to make $h^1(G_{\Q,S\cup\{p\}},\overline W)$ computationally accessible. Applying $\Hom_G(\cdot,\overline W)$ to (\ref{SES-WtOne2}) yields
\[
\xymatrix{
0\ar[r]& H^1(G_{\Q,S\cup\{p\}},\overline W)\ar[r]&
\Hom_G(\cO_E/p\cO_E,\overline W) \ar[r]& \Hom_G(\cO_E^\times ,\overline W).
}
\]
From (\ref{eq:FromBoston}) we deduce $\dim_{\F_p} \Hom_G(\cO_E/p\cO_E,\overline W) =2$ and $\dim_{\F_p}\Hom_G(\cO_E^\times ,\overline W)=1$. Since $p$ is inert in $K/\Q$ the completion $\cO_{K_p}$ of $\cO_K$ at $p$ carries a natural action of $G_p$, and via this identification there is a natural action of $G_p$ on $\cO_K/p\cO_K$, and by (\ref{eq:FromBoston}) one has $\cO_K/p\cO_K\cong\F_p[G_p]$. This yields $\cO_E/p\cO_E\cong\Ind_{G_p}^G\cO_K/p\cO_K$ as $\F_p[G]$-modules; there are two places of $E$ above $p$ and the respective local fields are isomorphic to $K_p=\cO_{K_p}[\frac1p]$. This gives
\begin{equation}\label{eq:Shapiro}
\xymatrix{\Hom_G(\cO_E/p\cO_E,\overline W) \ar[r]^\simeq& \Hom_{G_p}(\cO_K/p\cO_K,\overline W)} .
\end{equation}
To understand the isomorphism more explicitly, let $k_p$ be the residue field of $\cO_{K_p}$. Then $\cO_E/p\cO_E=k_p\times k_p$ and $\tau\in G\cong S_3$ interchanges the two factors. One verifies from the definition of induction that if $\psi\colon k_p\to \overline W$ is an $\F_p[G_p]$-homomorphism, then 
\begin{equation}\label{eq:IndRep}
\big(\Ind_{G_p}^G\psi\big)(a,b)=\psi(a)+\tau\psi(b) \qquad\hbox{for }(a,b)\in k_p\times k_p. 
\end{equation}

We also have the natural inclusion $\cO_K^\times\to\cO_E^\times$. Since $E$ contains no primitive $p$-th root of unity and since $K$ is the fixed field in $E$ under $\tau\in S_3$ via the $G$-action on $\cO_E^\times$, the restriction map
\begin{equation}\label{eq:UnitsEtoK}
\xymatrix{\Hom_G(\cO_E^\times ,\overline W) \ar[r]& \Hom(\cO_K^\times,\overline W^\tau)}
\end{equation}
into the $\tau$-invariants of $\overline W$ is an isomorphism, as well. This gives a short exact sequence
\begin{equation}\label{eq:MainMap}
\xymatrix{
0\ar[r]& H^1(G_{\Q,S\cup\{p\}},\overline W)\ar[r]&
\Hom_{G_p}(\cO_K/p\cO_K,\overline W) \ar[r]^-{\alpha_p^\vee}&  \Hom(\cO_K^\times,\overline W^\tau).
}
\end{equation}
To identify the map on the right, denote by $\epsilon_K\in\cO_K^\times$ a fundamental unit. Then $\epsilon_K^{p^3-1}$ lies in $1+p\cO_K$ and so we can write 
\begin{equation}\label{eq:DefZandEps}
\epsilon_K^{p^3-1}=1+zp \mod p^2\cO_K \qquad\hbox{for a unique }z\in \cO_K/p\cO_K=k_p. 
\end{equation}
Now embedding $\cO_K^\times$ into $\cO_E^\times$ and then passing to $\cO_E/p\cO_E=k_p\times k_p$, one finds that $\eps_K$ is mapped to the diagonal element $(z,z)$. Take on the other hand $\psi\in\Hom_{G_p}(k_p,\overline W)$. Then to describe the map we are interested in, we have to evaluate $\Ind_{G_p}^G\psi$ at $(z,z)$. By (\ref{eq:IndRep}) it follows that $\alpha_p^\vee(\psi)\in \Hom(\cO_K^\times,\overline W^\tau)$ is the unique map characterized by
\begin{equation}\label{eq:EquationOfMainMap}
\epsilon_K\mapsto (1+\tau)\psi(z) .
\end{equation}
Now $\alpha_p^\vee$ is the zero map if for all $\psi\in \Hom_{G_p}(k_p,\overline W)$ we have $(1+\tau)\psi(z)=0$. 
\begin{lemma}\label{lem:zEqZero}
$\alpha_p^\vee=0$ if and only if $z=0$.
\end{lemma}
\begin{proof}
One direction is obvious. For the other direction observe first that $k_p\cong \F_p\oplus \overline W$ as $\F_p[G_p]$-modules, where indeed $\F_p$ is the prime subfield $\F_p$ of $k_p$. We distinguish two cases.

\medskip

If $p\equiv 2\mod 3$, then $\overline W$ is irreducible and $\Aut_{G_p}(\overline W)\cong \F_{p^2}^\times$. Assume $z\in k_p\setminus \F_p$. Using the transitive action of $\Aut_{G_p}(\overline W)$ one can find $\psi\in \Hom_{G_p}(k_p,\overline W)$ with $0\neq \psi(z)$ and $\psi(z)$ in the $1$-dimensional subspace $\overline W{}^{\tau=-1}$. Then $(1+\tau)\psi(z)=2\psi(z)\neq0$, which violates $\alpha_p^\vee=0$. Thus $z\in\F_p$. Now it follows from the first isomorphism in (\ref{eq:FromBoston}) that~$z=0$.

\medskip

If $p\equiv1\pmod3$, then $W=U\oplus U'$ for one-dimensional irreducible representations of $G_p$: the element $\sigma\in G_p$ acts as multiplication by a primitive third root of unity, and both choices occur. Suppose $U$ is the submodule of matrices of the form $\MatTwo{0}{b}{0}{0}$ under the adjoint action and let $\psi=\id_U\times 0_{U'}$. Then $(1+\tau)\psi$ of any such matrix is $\MatTwo{0}{b}{b}{0}$. This shows that $z$ must have component $0$ in $U$. An analogous argument shows that $z$ has component $0$ in $U'$, and hence we find $z\in\F_p$. As in the previous case this implies~$z=0$.
\end{proof}

We checked the condition $z=0$ for imaginary cubic fields whose discriminant $\Delta$ lies in the range $-140 \leq \Delta < 0$;\footnote{We refer to \cite[App.~B.3 and B.4]{Cohen} for extensive tables of complex cubic and real cubic number fields.} hence 
\[\Delta\in\{-23,-31,-44,-59,-76,-83,-87,-104,-107,-108,-116,-135,-139,-140\}.\]
We looked specifically at primes $p$ such that $3<p \leq 10^8$ with $\bar{\rho}(\Frob_p)$ of order $3$. Our computation (realized in Magma code) give the following result:
\begin{proposition}\label{prop:TestForB}
Let $\Delta$ be in the above list. Suppose that $p$ is in the above range, that $p$ does not divide $\Delta$ (and is hence unramified in $E/\Q$) and that $p$ is not in the list $H_5$ from  (\ref{eq:DefH5}) for $S=\{l\mid l\hbox{ is prime and }l|\Delta\}$, so that $p$ satisfies (Hyp 1) to (Hyp 5). Then 
\[h^2(G_{\Q,S\cup\{p\}},\Ad^0(\bar\rho))=0.\]
\end{proposition}
For completeness, we list the sets $H_5\setminus\{2,3\}$. 
\begin{table}[htb]
\begin{center}
\setlength{\tabcolsep}{0.15em}
\begin{tabular}{c|c|c|c|c|c|c|c|c|c|c|c|c|c|c|c|c|c|}
$\Delta$ & $-23$& $-31$ & $-44$ &$-59$  & $-76$ &$-83$ &$-87$ & $-104$ & $-107$ &$-108$ &$-116$ &$-135$ &$-139$&-140 \\
\hline
$H_5$ &$\{11\}$&$\{5\}$&$\{5\}$&$\{5,29\}$&$\{5\}$&$\{7,41\}$&$\{5,7\}$&$\{7\}$&$\{53\}$&$\emptyset$&$\{5,7\}$&$\emptyset$&$\{5,7,23\}$&$\emptyset$\\
\end{tabular}
\end{center}\caption{}\label{MinusOne}
\end{table}
}

It might be instructive to also consider analogous examples where $K$ is a real cubic field, so that one is in case (B). The computations would be somewhat more involved since here the units of $K$ have rank  $2$ and the computation of the map $\alpha_p$ from (\ref{eq:random-map}) is more demanding. In case (B) the conjecture predicts that for any discriminant $\Delta$ one has infinitely many exceptional $p$, where however the size of these $p$ grows doubly exponentially. So it is not clear how visible this is in numerical experiments.

\section{Weight one $p$-adic forms} \label{subsec:WeightOne}
We now ask  a related question about $p$-adic weight one forms which addresses ramification at $p$
in mod $p$ representations coming from Galois representations attached to $p$-adic weight one forms.   (In the earlier section we considered ramification at $q \neq p$ mainly of mod $p$ represenations arising from weight 2 forms.) Consider the space $S_{1,p\textrm{-adic}}(\Gamma_1(N),\Q_p)^\ord$ of ordinary overconvergent $p$-adic weight one forms. Let $\rho$ be a representation which arises from a classical weight $1$ form $f$ of level $N$, and denote by $S_{1,p\textrm{-adic}}(\Gamma_1(N),\Q_p)^\ord_{\rho\!\!\mod p}$ the subspace of $S_{1,p\textrm{-adic}}(\Gamma_1(N),\Q_p)^\ord$ of forms congruent to $f$.  Suppose $\rho$ has nebentype $\eps$. Let $\chi_p$ be the Teichm\"uller lift of the mod $p$ cyclotomic character at $p$. By Hida theory and since any mod $p$ modular form of weight strictly larger than $1$ lifts to characteristic zero with the same level, one sees that $\rho\mod p$ also arises from the classical spaces of forms $S_p(\Gamma_1(N),\Q_p)$, as well as $S_2(\Gamma_1(N)\cap\Gamma_0(p),\chi^{-1}_p\eps,\Q_p)$ -- overconvergent ordinary forms for the above two levels and nebentypes are classical for weight $k\ge2$. 

Consider  primes $p$ such that $\rho_p({\rm Frob}_p)$ has distinct eigenvalues $\alpha_p,\beta_p$, i.e., $\rho_p$ is distingushed at $p$. For such $p$ we localize the space of mod $p$ modular forms at the maximal ideal $\sf m_\alpha$ (resp. $\sf m_\beta$) which gives rise to $\rho$ and contains $U_p-\alpha_p$ (resp. $U_p-\beta_p$). We investigate the following question:

\begin{question} For primes $p$ at which $\rho_p$ is distinguished,  is one of  $\dim S_{p}(\Gamma_1(N),\Z_p)_{\sf m_\alpha}$ or  $\dim S_{p}(\Gamma_1(N),\Z_p)_{\sf m_\beta}$ of $\Z_p$-rank one for almost all primes p?
\end{question}

An affirmative answer to this question implies that the  ordinary deformation ring (for a choice of $\alpha_p,\beta_p$) of ordinary deformations of $\rho_p$ is smooth and thus implies using thePoitou-Tate sequence  that $H^2(S \cup\{p\},{\rm Ad}(\rho_p))=0$ for almost all primes $p$. Thus this question is a refinement of the question we studied in the previous section.

In a similar spirit we could ask:
\begin{question}\label{Que3} 
Is $\dim S_{1,p\textrm{-adic}}(\Gamma_1(N),\Q_p)^\ord_{\rho\!\!\mod p}$ bounded as $p$ varies?
\end{question}

 Hence one can reformulate Question~\ref{Que3} as: 

\begin{question}\label{Que4}  Is the following dimension bounded as $p$ varies\\
\[\dim S_{p}(\Gamma_1(N),\Q_p)^\ord_{\rho\!\!\!\mod p}=\dim S_{2}(\Gamma_1(N)\cap\Gamma_0(p),\chi^{-1}_p\eps,\Q_p)^\ord_{\rho\!\!\!\mod p} ?\]
\end{question}
\medskip
 
A somewhat more general question is the following:
\begin{question} Is the subspace of $\dim S_{p}(\Gamma_1(N),\Q_p)^\ord$ spanned by eigenforms whose mod $p$ Galois representation is unramified at $p$ of bounded dimension as $p$ varies?
\end{question}

{ \subsection{Computational data via modular forms}\label{subsec:ModFormsComp}
In this subsection, we explain how we investigated Question~\ref{Que4} numerically via computations with modular forms in one particular example. In Subsection~\ref{ExampleCrefined} we shall explain a more efficient method due to F. Calegari and apply it to all examples from Subsection~\ref{ExampleC}.}
\smallskip

Let $p=23$ and let $f$ be the unique dihedral weight $1$ form of level 23. Attached to $f$ there is a complex odd $2$-dimensional Galois representation $\rho\colon G_\Q\to\GL_2(\C)$ with image isomorphic to the dihedral group $S_3$. The fixed field $E=\overline\Q\,{}^{\ker\rho}$ is in fact the Hilbert class field of the quadratic extension $\Q(\sqrt{-23})$. In particular, $\rho$ is only ramified at $23$. Note also that $\rho$ is the Teichm\"uller lift of the mod $23$ reduction $\bar\rho_{\Delta,23}$ of the $23$-adic Galois representation attached to Ramanujan's $\Delta$-function. The last observation can be used to efficiently compute the coefficients $a_l(f)$ for primes $l$ that are not too large -- certainly for all $l$ that occurred in our sampling. Lifting the coefficients is simple since $a_l\in\{-1,0,2\}$ for $l\neq 23$ and $a_{23}(f)=1$. 

\medskip

Using Magma, we first computed the intersection in $S_p:=S_{2}(\Gamma_1(N)\cap\Gamma_0(p),\chi^{-1}_p\eps,\F_p)^\ord$ of the kernels of $(T_l-a_l(f))^2$ mod $p$ where $l$ ranges over the $5$ primes in $\{5,\ldots,17\}$, excluding $p$, for the $139$ primes $p$ from $3$ to~$809$. The computation revealed 
\[
 \dim S_{2}(\Gamma_1(N)\cap\Gamma_0(p),\chi^{-1}_p\eps,\F_p)^\ord_{\rho\!\!\mod p} =2, 
\]
for all primes $p$ in this range for which $\rho(\Frob_p)$ was non-scalar, i.e., not the identity, with the single exception $p=13$. For $p=13$ and for those primes for which $\rho(\Frob_p)$ is scalar, the multiplicity was at least $3$. In fact, for $p=13$ we obtained $\dim (S_p)_{\rho\mod p}=3$, for those $p$ such that $\rho(\Frob_p)$ was trivial, we found $\dim (S_p)_{\rho\mod p}=4$ in all cases.

\smallskip 
Experimentally, we also observed that $\dim S_p$ roughly grows like $2p$ and that the difference $d_p:=\dim S_{2}(\Gamma_1(N)\cap\Gamma_0(p),\chi^{-1}_p\eps,\Q_p)-\dim S_p$ appears to be bounded. In fact, R. Pollack pointed out to G.B., that in all cases we considered the difference $d_p$ is due to the presence of CM forms in $S_p$. All other forms were ordinary. In particular, the boundedness asked for in Question~\ref{Que3} is non-trivial. \smallskip

Let us first recall what the theory tells us: Let $p>2$ be a prime. If $\rho(\Frob_p)$ is non-scalar, then each of the two eigenvalues of $\rho(\Frob_p)$ is a possible $U_p$ eigenvalue for the Hecke action on $S_p$ (by  Hida theory). Now by multiplicity one, due to Mazur, Mazur-Ribet, Gross, Edixhoven and finally Buzzard -- for references see \cite[Rem.~4.9]{CG} --, it follows that the subspace of $S_p$ defined as 
\[S'_p:=\bigcap_{l\neq 23,p} \ker(T_l-a_l(f))\cap \ker (U_p^2-a_p(f)U_p+\eps(p))\] 
has dimension exactly $2$. If on the other hand $\rho(\Frob_p)$ is scalar, then it follows by multiplicity two proved in \cite[Thm.~4.8]{CG} and anticipated in \cite{Wiese}, where multiplicity at least two was proved, that the subspace $S'_p$ has again dimension $2$. The results just quoted also explain why in our experiment we compute the annihilator of $\ffrm^2$ and not that of $\ffrm$, where $\ffrm$ denotes the maximal ideal of the Hecke algebra of $S_p$ corresponding to $\rho$. \medskip

We now provide a heuristic for the above findings for those primes $p$ at which $\rho$ is unramified and at which $\rho(\Frob_p)$ is non-scalar. Similar to Subsection~\ref{subsec:WeightTwo}, we have no heuristic for the other $p$. Our heuristic explanation of the numerical results given above is a combination of the predictions from Section~\ref{sec:Artin} and an adaption of the heuristic from Subsection~\ref{subsec:WeightTwo}: First, from Proposition~\ref{prop:6.5}(iii), assuming Heuristic~\ref{HeuristicAlpha}, we deduce that
\[ H^2(G_{S(23)\cup\{p\}},\Ad^0(\rho\!\!\mod p))=0\qquad\hbox{for almost all primes }p.\]
By the Euler characteristic formula in Galois cohomology, it follows for almost all primes $p$ that $\dim H^1(G_{S(23)\cup\{p\}},\Ad^0(\rho\!\!\mod p))=2$. In the example we investigated, Proposition~\ref{prop:TestForB} yields $\dim H^1(G_{S(23)\cup\{p\}},\Ad^0(\rho\!\!\mod p))=2$ for all primes $p$ with $3<p \leq 10^8$.

\medskip

Before we continue, let us first recall some local computations of cohomology and Selmer groups: If $p>3$ and if $\rho(\Frob_p)$ is non-scalar, then one easily shows from local duality that $H^2(G_{\Q_p},\Ad^0(\rho\mod p))=0$; it follows that $\dim H^1(G_{\Q_p},\Ad^0(\rho\mod p))=4$. The explicit formulas given in \cite[Lemma~5]{Ramakrishna-FM} show $\dim H^1_\ord(G_{\Q_p},\Ad^0(\rho\mod p))=2$. We now formulate the an analog of Heuristic~\ref{Heuristic-RibetsQ}:
\begin{heuristic}\label{Heuristic-WeightOne}
The sum of the restriction map
\[ r_q\colon H^1(G_{S(23)\cup\{p\}},\Ad^0(\rho\!\!\mod p)) \longrightarrow  H^1(G_{\Q_p},\Ad^0(\bar{\rho}))\]
and the inclusion
\[\iota_p^\ord\colon H^1_\ord(G_{\Q_p},\Ad^0(\rho\mod p)) \longrightarrow H^1(G_{\Q_p},\Ad^0(\bar{\rho})) \]
is random in $p$ as a map of $\F_p$ vector spaces $\F_p^4\to\F_p^4$ subject to the hypothesis that $r_p$ and $\iota_p$ are injective.
\end{heuristic}
Now if $\rho\mod p$ is $p$-distinguished, i.e., if $(\rho\mod p)(\Frob_p)$ has two distinct eigenvalues, then Hida theory gives two weight $2$ form congruent mod $p$ to $\rho$ whose $U_p$ mod $p$ eigenvalues are the two eigenvalues of $(\rho\mod p)(\Frob_p)$. For each of these we consider $\iota^\ord_p\oplus r_p$. The kernel of this map is the mod $p$ Selmer group for ordinary deformations of $\rho\mod p$. Hence this map is injective if and only of the corresponding universal ordinary deformation ring is a quotient of $\Z_p$. Since the ring has a characteristic zero point, it then must be equal to $\Z_p$. Thus the injectivity applies to both weight $2$ forms if and only if $\dim (S_p)_{\rho\!\!\mod p}=2$. Now chances that $\iota_p^\ord\oplus r_p$ is injective under the condition that the individual maps are injective corresponds to choosing two times $2$ linearly independent vectors in $\F_p$ and asking whether the combined $4$ chosen vectors are linearly independent. The probability for this~is
\[\frac{(p^4-p^3)(p^4-p^2)}{(p^4-1)(p^4-p)}=\frac{p^4}{(p^2+p+1)(p^2+1)} \approx1-\frac1p,\] 
where on the left we already cancelled the contribution of one choice of $2$ linearly independent vectors. Arguing as in the proof of Proposition~\ref{prop:6.5}, this gives the following conclusion:
\begin{proposition}\label{prop:PredictionWtOne}
Assuming Heuristics~\ref{HeuristicAlpha} and \ref{Heuristic-WeightOne}, 
\begin{enumerate}
\item for a prime $p>3$ for which $\rho(\Frob_p)$ is non-scalar, the probability for the event
\[\dim (S_p)_{\rho\!\!\!\!\mod p}>2\]
is $1/p$
\item there are approximately $\log \log X$ primes $p \leq X$ for which (1) holds.
\end{enumerate}
\end{proposition}

\begin{remark}
Let us compare the above to the case where $f$ is a cuspidal Hecke eigenform of weight $k\ge2$ (and new at level $N$) with corresponding maximal ideal $\ffrm$. In this case there can be only finitely many primes $l$ of $\Z$ for which there is a cuspidal Hecke eigenform $g\in S_k(\Gamma_0(N))$ such that $f$ and $g$ are congruent modulo $l$. This is well-known and follows from the fact that the Hecke algebra $\T$ over $\Z$ for $S_k(\Gamma_0(N))$ is an order in the finite product of number fields $\T\otimes_\Z\Q$ and hence for $l\gg0$ the completion of $\T$ at $l$ is a product of discrete valuation rings unramified over~$\Z_l$. Hence in this case, there are only finitely many primes $p$ where $\dim S_k(\Gamma_0(N))_{\ffrm}$ is larger than the expected value~$1$. The basic distinction between weights $k=1$ and $k \geq 2$ is that by Hida's theorem all forms in the $p$-adic space of cusp forms $S_k(\Gamma_1(N),\Q_p)^{ord}$ are classical for $k \geq 2$, unlike the case of $k=1$.
\end{remark}

{ \subsection{Computational data via class field theory} \label{ExampleCrefined} 
We now give  a method due to F. Calegari to detect when 
\[ \dim S_{p}(\Gamma_1(N),\Q_p)^\ord_{\rho\!\!\!\mod p} > 2\]
for not only the representation $\rho$ considered in Subsection~\ref{subsec:ModFormsComp} but all $\rho$ considered in Subsection~\ref{ExampleC}.
\smallskip

So we resume the setup of Subsection~\ref{ExampleC}; so $W$ is a vector space over $\C$ of dimension $2$, $\rho\colon G_\Q\to \Aut(W)$ is a representation with image $S_3$ whose splitting field $E$ is the Galois closure of a complex cubic field $K$, and $\chi=\det\rho$. Throughout, $p$ will be a prime that satisfies (Hyp 1) to (Hyp 5) from Subsection~\ref{ExampleC} and such that $\bar{\rho}(\Frob_p)$ has order $3$, so that $\chi(\Frob_p) = 1$ and that $p$ is inert in $\cO_K$. For the reductions of $W$ and $\chi$ mod $p$ we write $\overline W$ and $\bar \chi$. We assume further that $h^1(G_{\Q,S\cup\{p\}},\overline W) = 1$, or equivalently, as observed in Subsection~\ref{ExampleC}, that $h^2(G_{\Q,S\cup\{p\}},\overline W) = 0$. Recall also that $\Ad^0(\rho)\cong W\oplus \chi$. 

\medskip

We now study whether the ordinary weight $2$ subspace $H^1_\ord(G_{\Q,S\cup\{p\}},\Ad^0(\bar\rho))$ of the two-dimensional $F_p$ vector space $H^1(G_{\Q,S\cup\{p\}},\Ad^0(\bar\rho))$ vanishes. For this to make sense over $\F_p$, we assume that $p\equiv 1\mod 3$. Then $\overline W|_{G_p}$ is a direct sum of two non-isomorphic $\F_p[G_p]$-submoduls $U,U'$ on which $\Frob_p$ acts with order three. Let $U_0\subset \Ad^0(\bar\rho)\subset\End_{\F_p}(\overline W)$ be either the subspace mapping $U$ to $0$ and $U'$ to $U$, or that mapping $U'$ to $0$ and $U$ to $U'$. Set
\[H^1_\ord(G_{\Q_p},\Ad^0(\bar\rho))=\ker \big( H^1(G_{\Q_p},\Ad^0(\bar\rho)) \to H^1(I_p,\Ad^0(\bar\rho)/U_0)^{G_{\Q_p}}\big),\]
and define $H^1_\ord(G_{\Q,S\cup\{p\}},\Ad^0(\bar\rho))$ as the subspace of classes in $H^1(G_{\Q,S\cup\{p\}},\Ad^0(\bar\rho))$ whose restriction to $H^1(G_{\Q_p,},\Ad^0(\bar\rho))$ lie in $H^1_\ord(G_{\Q_p},\Ad^0(\bar\rho))$; cf.~\cite[p.~124ff.]{Ramakrishna-FM}. 

\begin{remark} If $p \equiv 2 \Mod 3$ and $\Frob_p$ has order $3$, then the eigenvalues of $\bar{\rho}(\Frob_p)$ will not be defined over $\F_p$. It follows that the ordinary subspace is not a pair of lines in $H^1(G_{\Q_p},\Ad^0(\bar\rho)(\bar{\rho}))$ but rather a pair of conjugate lines in $H^1(G_{\Q_p},\Ad^0(\bar{\rho})\otimes \F_{p^2})$. In particular, the vanishing of $H^2(G_{\Q,S\cup\{p\}},\Ad^0(\bar{\rho}))$ implies that the ordinary weight one deformation ring will be $\Z_{p^2}$ (for deformations to $\Z_{p^2}$-algebras).
\end{remark}

Under the identification $\Ad^0(\bar\rho)=\overline W\oplus\bar\chi$, the subspace $U_0$ is identified with a subrepresentation of $\overline W$, namely $U'$ in the first case and $U$ in the second. I.e., the two ordinary lines in $\Ad^0(\bar\rho)$ correspond to the two eigenspaces of $\overline W$ under the action of $G_p$. We set 
\[H^1_\ord(G_{\Q_p},\overline W)=\ker \big( H^1(G_{\Q_p},\overline W) \to H^1(I_p,\overline W/U_0)^{G_{\Q_p}}\big),\]
and define $H^1_\ord(G_{\Q,S\cup\{p\}},\overline W)$ as the kernel of $H^1(G_{\Q,S\cup\{p\}},\overline W)\to H^1(I_p,\overline W/U_0)$ under iterated restriction. We need to understand the vanishing of $H^1_\ord(G_{\Q,S\cup\{p\}},\overline W)$. Using inflation-restriction as in Subsection~\ref{ExampleC} and local class field theory, we have isomorphisms 
\[ H^1(G_{\Q_p},\overline W)\cong (\Hom(G_{E_\frp},\F_p)\otimes \overline W)^{G_p}\cong (\Hom(E^\times_\frp,\F_p)\otimes \overline W)^{G_p} \] 
for $\frp$ one of the two places of $E$ above $p$, and we note that under our hypotheses we may identify $E_\frp$ with $K_p$. Note also that the kernel of $\Hom(E^\times_\frp,\F_p)\to \Hom(\cO_{E_\frp}^\times,\F_p)$ is the subspace of unramified classes in $H^1(G_{\Q_p},\F_p)$. Using the local-global compatibility of class field theory, we can extend the sequence (\ref{SES-WtOne}) to a commutative diagram with exact rows
\begin{equation}\nonumber
\xymatrix@C-.6pc{
\cO_E^\times\otimes_\Z\F_p \ar[r]\ar[d]&  \cO_{E_p}^\times\otimes_\Z\F_p \ar[r]\ar@{=}[d]& \ar[d] H^1(G_{E,\{p,\infty\}},\F_p)^{\vee} \ar[r]& 0&\\
0\ar[r]&  \bigoplus_{\frp|p}  \cO_{E_\frp}^\times\otimes_\Z\F_p \ar[r]& \bigoplus_{\frp|p} H^1(G_{E_{\frp}},\F_p)^{\vee} \ar[r]& \bigoplus_{\frp|p} H^1_\unr (G_{E_{\frp}},\F_p)^{\vee} \ar[r]& 0
}
\end{equation}
We dualize (over $\F_p$) and apply $\Hom_G(\cdot ,\overline W)$. Using the identification from (\ref{eq:Shapiro}) and the definition of $H^1_\ord(G_{\Q_p},\overline W)$, one obtains the following commutative diagram with exact rows and columns
\begin{equation}\nonumber
\xymatrix@C-1pc @R-.5pc {
&&&0\ar[d]&\\
&&&\Hom_{G_p}(\cO_{K_p}^\times,U_0)\ar[d]\ar@{-->}[r]&\Hom_{G_p}(\cO_{K}^\times,\overline W^\tau) \\
&0\ar[r]&H^1(G_{\Q,S\cup\{p\}},\overline W)\ar[r]\ar@{=}[d]&\Hom_{G}(\cO_{E_p}^\times,\overline W)\ar[d]\ar[r]&\Hom_{G}(\cO_{E}^\times,\overline W)\ar[u]^\simeq_{(\ref{eq:UnitsEtoK})}\\
0\ar[r]&H^1_\ord(G_{\Q,S\cup\{p\}},\overline W)\ar[r]&H^1(G_{\Q,S\cup\{p\}},\overline W)\ar[r]^\beta&\ar[d]\Hom_{G_p}(\cO_{K_p}^\times,\overline W/U_0)\\
&&&0&\\
}
\end{equation}
From the diagram we deduce that $H^1_\ord(G_{\Q,S\cup\{p\}},\overline W)=0$ if and only if $\beta$ is injective, and the latter is equivalent to the dashed map being injective. This map is the map $\alpha_p^\vee$ from (\ref{eq:MainMap}) restricted to $\Hom_{G_p}(\cO_{K_p}^\times,U_0)$. Let $\epsilon_K$ and $z$ be as in (\ref{eq:DefZandEps}). Arguing as in the proof of Lemma~\ref{lem:zEqZero} one shows:
\begin{lemma}
Suppose that $h^1(G_{\Q,S\cup\{p\}},\overline W)=1$, i.e, that $z\neq0$. Then $\alpha_p^\vee$ restricted to $\Hom_{G_p}(\cO_{K_p}^\times,U_0)$ vanishes if and only if $z$ is an eigenvalue under the action of $\Frob_p$ (with eigenvalue necessarily a primitive third root of unity of $\F_p$). 
\end{lemma}
Now $\Frob_p z=\lambda z$ for $\lambda\in\F_p$ a primitive third root of unity is equivalent to $z^{3(p-1)}=1$.} Hence we need to examine when 
\[\left(\frac{\epsilon_K^{p^3-1}-1}{p} \right)^{3(p-1)} \equiv 1 \Mod p.\]

\medskip

We checked this equivalence for imaginary cubic fields {  whose discriminant $\Delta$ lies in the range} $-140 \leq \Delta < 0$. We looked at the same primes $p$ with $3<p<10^8$ with $\bar{\rho}(\Frob_p)$ of order $3$ as in Proposition~\ref{prop:TestForB}, omitting those with $p\equiv 2 \Mod 3$. The results are shown in Table~\ref{TableOne}. We omitted the prime $7$ for $\Delta=-139$ since it lies in $H_5$ given in Table~\ref{MinusOne}.  
\begin{table}[htb]
\begin{center}
\begin{tabular}{c|c}
$\Delta$ & $p$ \\
\hline
$-23$ & $13$ \\
$-31$ & $7$, $2467$ \\
$-44$ & \\
$-59$ & $19$\\
$-76$ & $125743$\\
$-83$ & $7$, $31$ \\
$-87$ & $181$ \\
$-104$ & $12697$ \\
$-107$ & $13$ \\
$-108$ & $3511$ \\
$-116$ &  \\
$-135$ & \\
$-139$ & $31$ \\
$-140$ & \\
\end{tabular}
\end{center}\caption{}\label{TableOne}
\end{table} \medskip

\section{Almost multiplicity one mod $p$ and level raising at primes $q \neq p$}

In this section we deal with ramification  at primes away from $p$    of mod $p$ representations arising from mod $p$ newforms. Thus the question is different from the questions about unobstructedness we investigated earlier, but the guiding heuristic is  very similar, and is also related to the earlier section where we considered ramification at $p$ of mod $p$ representations arising from weight one $p$-adic forms.\medskip

\subsection{Two guiding expectations}
Let us begin by explaining some general expectations before we investigate in more detail two concrete cases. Denote by $S_k(N,p)$ the space of weight $k$ modular forms of level $\Gamma_0(N)$ with coefficients in $\overline{\bbF_p}$. We consider it as a Hecke module for the tame level $N$ Hecke algebra generated over $\overline{\F_p}$ by the Hecke operators $T_p$ where $p$ ranges through all primes not dividing $N$. We semisimplify the action of the Hecke algebra, since this need not hold in characteristic $p$. To any Hecke eigenform $f\in S_k(N,p)$ one can attach a continuous two-dimensional semisimple Galois representation $\bar\rho_f\colon G_{\Q,S(N)\cup\{p\}}\to \GL_2(\overline\F_p)$ that is completely characterized by $\bar\rho_f(\Frob_p)=a_p(f)$ for all primes $p$ not dividing $N$, where $a_p(f)$ is the eigenvalue of $f$ under $T_p$ and $S(N)$ is the set of places of $\Q$ dividing $N$ or~$\infty$.\medskip

The first expectation is that the vast majority of the representations $\bar\rho_f$ is irreducible and that the reducible ones asymptotically have density zero. For simplicity we call such $f$ irreducible and the others reducible. We do not attempt to quantify this expectation. An important consequence of this expectation is that to the vast majority of forms of $S_k(N,p)$ one can apply the level lowering and level raising results of Ribet et al; see \cite{Ribet}.\medskip

To formulate the second expectation, note first that because of old forms, the space $S_k(N,p)$ cannot satisfy a multiplicity one hypothesis; i.e., the generalized Hecke eigenspaces may have multiplicity larger than one. However even if one takes this into account, multiplicity one will fail in general: the theory of level raising and level lowering predicts under suitable hypothesis congruences between old and new forms.\medskip

The second expectation is that the multiplicity of any generalized Hecke eigenspace of $S_k(N,p)$ for  which $\bar\rho_f$ is irreducible, is up to a small error, the minimal dimension predicted by the theory of new and old forms combined with the theory of congruences for level changing of mod $p$ forms of Ribet et al. \medskip

An important consequence of the above two expectations is that ramification  should be abundant.  There is surprisingly little one can prove  towards this. One expects a Murphy's law for ramification in that it should occur (plentifully!)  whenever its allowed to occur. This is hard to prove in  the presence of unramified forms $g\in S_k(N',p)$ of level $N'$ strictly dividing $N$ which conjecturally should inhibit ramification but only in statistically negligible ways! For instance  for any conductor $N$ or weight $k$ not too small, there should typically be a modular mod $p$ Galois representation $\bar\rho\colon G_{S(N)\cup\{p\}}\to\GL_2(\overline{\F_p})$ of weight $k$ which is not finite at $p$ and ramified at all primes dividing $N$. The argument under the above expectation is as follows: By the first expectation almost all eigenforms $f$ in $S_k(N,p)$ are irreducible. Now if $\bar\rho_f$ is finite at $p$ ore unramified at some divisor of $N$, then by level lowering results $\bar\rho_f=\bar\rho_g$ for some Hecke eigenform $g\in S_k(N',p)$ of level $N'$ strictly dividing $N$. To conclude, it is now necessary that the dimension of all irreducible forms in $S_k(N,p)$ that admit level lowering is strictly less than the dimension of $S_k(N,p)$. This typically follows from the second expectation.\medskip

The following simpler question seems already hard to answer. Fix primes $p$ and $l$, that are possibly equal, then is it true that for all but finitely many primes  $q$  there is an irreducible form $f$ in $S_2(lq,p)$ such that $\bar\rho_f$ is ramified at $q$ and not finite at $l$ (i.e., tr\`es ramifi\'ee when $l=p$ and simply ramified  at $l$ when $l \neq p$)? Note that we insist, because of the results of the next paragraph, that the representation be not finite at~$l$.\medskip

There is one easy case precisely in the situation that we do not have to contend with unramified  forms. Namely, if   $N$ is a prime $q$, it is easy to show (by a local argument)  that if $q$ is not $\pm 1 $ mod $p$, any representation arising from $S_2(q,p)$ is irreducible,  and  then it is ramified at $q$ by level lowering results of Ribet. (In fact, as Romyar Sharifi indicated to us, a global argument shows that if $q$ is not $1$ mod $p>3$ then any representation arising from $S_2(q,p)$ is irreducible.) In accordance with the above expectations, one should ask if for a fixed $p>3$ and varying $q$ the Hecke module of forms in $S_2(q,p)$ is almost multiplicity free in the sense that the number of mod $p$ eigenforms arising from it is roughly its dimension. \medskip

In the following two subsections we shall explain two related numerical experiments that rely in parts on a (well-known) deformation theoretic interpretation of level raising.

\subsection{Weight 2 Forms}\label{subsec:WeightTwo}
 Let $\bar{\rho}: G_{\Q} \to \GL_2(k)$ be an absolutely irreducible representation, where $k$ is a finite field of characteristic $p$.  We assume throughout this subsection that $\Ad^0(\rhobar)$ is also irreduicble as the case of dihedral $\rhobar$ has a different flavor which we treat  separately. Suppose $\bar{\rho}$ arises from $S_2(\Gamma_0(N))$, corresponding to a maximal ideal $\I{m}$ of the Hecke algebra acting on this space. By the \v{C}ebotarev density theorem, the set $X$ of primes $q$ such that $\bar{\rho}(\Frob_q)$ has eigenvalues with ratio $q$ has positive density. Let $q \in X$, and let $\I{m}'$ denote the maximal ideal of the Hecke algebra acting on $S_2(\Gamma_0(Nq))$, which is the same as $\I{m}$ away from $q$, and such that $U_q^2-1 \in \I{m}'$. In this case level raising asserts that there is a Hecke eigenform in $S_2(\Gamma_0(Nq))$ that is new at $q$ and whose associated Galois representation mod $p$ is equal to $\bar\rho$, i.e., that there is a congruence. The question we want to ask is the following:
\begin{question}\label{lift} Is $\dim S_2(\Gamma_0(Nq))^\qnew_{\I{m}'} = 1$ for a positive density  of primes $q \in X$?\footnote{For $q\equiv-1\pmod p$ the expected dimension has to be $2$ and not $1$ by the discussion below.}
\end{question}
Similarly, we could ask:
\begin{question} 
What is the probability of $\dim S_2(\Gamma_0(Nq))^\qnew_{\I{m}'}=k$ for fixed $k\ge1$, as $q \in X$ varies?
\end{question}

We have analyzed the first  question for some small fixed $N$ and $p$. In fact, the first case considered was $N=11$ and $p=11$, building off numerical investigations carried out by Tommaso Centeleghe. In this case $S_2(\Gamma_0(11),\F_{11}) = S_{12}(\SL_2(\Z),\F_{11}) = \F_{11} \cdot \Delta$ as Hecke modules. Suppose $\bar{\rho} = \bar{\rho}_{\Delta,11}$, and consider primes $q$ for which $\tau(q) = \pm (q+1) \Mod 11$. Question \ref{lift} in this setting asks whether there is a unique $q$-new form in $S_2(\Gamma_0(11q))$ which gives rise to $\bar{\rho}$.\medskip

Fix a prime $p$ and a square free integer $N>1$, and let $f\in S_2(N,p)$ be a Hecke eigenform whose generalized Hecke eigenspace has dimension $1$ and such that $\bar\rho:=\bar\rho_f$ is not finite any prime dividing $N$. To interpret our findings, we first analyze level raising at a prime $q$ from a deformation theory perspective. Thus as before we consider primes $q$ for which Ribet's level raising condition $a_q(f) = \pm (q+1) \Mod p$ is satisfied, but ignore those primes for which $q \equiv 1 \Mod p$ and $\bar{\rho}(\Frob_q)$ is a scalar matrix. Let $k$ be the coefficient field of $f$, so that in the following we may and will regard $\bar\rho$ as a representation to $\GL_2(k)$.\medskip

There are up to four universal deformation rings $R^?$ relevant here: All rings parametrize deformations of $\bar{\rho}$ which are unramified outside $N \!pq$, and which at all primes $l$ dividing $N$ are of the form 
\begin{equation}\label{eq:shape-of-def}
\MatTwo{\chi_{p}}{*}{0}{1},
\end{equation}
where $\chi_{p}$ denotes the $p$-adic cyclotomic character, i.e., unramified Steinberg at $l$ if $l\neq p$ and ordinary weight $2$ if $l=p$, and which at $p$ are crystalline of weight $2$ if $p$ does not divide~$N$. The ring $R^\qunr$ parameterizes those deformations which are unramified at $q$, the ring $R^\qnew$ those which are of the form (\ref{eq:shape-of-def}) at $q$, i.e., which are Steinberg at $q$, and the ring $R$ those which are unrestricted at $q$. If $q\equiv -1\pmod p$, we also consider the ring $R^\qnewtw$ which at $q$ parameterizes deformations of the shape in (\ref{eq:shape-of-def}) twisted by $\chi_p$, i.e. unramified quadratic twists of Steinberg at $q$. (The reason for ignoring the primes $1 \Mod p$ for which $\bar{\rho}(\Frob_q)$ is scalar is that the local problem at $q$ may not be representable in this case.) We refer to \cite{Ramakrishna-FM} and to \cite{Boston-Ramif} for a detailed deformation theoretic discussion of the local deformation rings introduced above at primes dividing $Np$ and dividing $q$, respectively; in particular \cite{Ramakrishna-FM} provides the local computations of the dimensions displayed below.\medskip

For each of the above deformation problems, the tangent space is given by a Selmer group. At primes $l$ dividing $Np$, in all cases, the local problem is given by the same subspace $\Sc{L}_l \subset H^1(G_{\Q_l},\Ad^0(\bar{\rho}))$. At $q$ we have subspaces $\Sc{L}_q^\qunr$, $\Sc{L}_q^\qnew$ and $\Sc{L}_q^\qnewtw$ of $\Sc{L}_q=H^1(G_{\Q_q},\Ad^0(\bar{\rho}))$, and in each case we have a tuple $\Sc{L}^?=(\Sc{L}_l)_{l|Np}\cup(\Sc{L}_q^?)$ such that the tangent space of $R^?$ is the Selmer group $H^1_{\Sc{L}^?}(G_{\Q,S},\Ad^0(\bar{\rho}))$, where $S = S(Nqp)$. Note that $\Sc{L}_q^\qunr=H^1(G_{\Q_q}/I_q,\Ad^0(\bar{\rho}))$ where $I_q$ is the ramification subgroup of $G_{\Q_q}$. Assuming that $\rho(\Frob_q)$ is non-scalar if $q\equiv 1\pmod {p}$, we have the following dimensions of the local problems:
\begin{eqnarray*}
q \equiv -1 \Mod p&:& \dim H^1(G_{\Q_q},\Ad^0(\bar{\rho}))=3, \dim \Sc{L}_q^\qunr\!=\dim \Sc{L}_q^\qnew\!=\dim \Sc{L}_q^\qnewtw \!= 1,\\
q \not \equiv- 1 \Mod p&:&  \dim H^1(G_{\Q_q},\Ad^0(\bar{\rho})) = 2, \dim \Sc{L}_q^\qunr\!=\dim \Sc{L}_q^\qnew\!= 1.\end{eqnarray*}

In each of the above cases, an isomorphism $R^?\cong \T^?$ is known for a corresponding Hecke algebra $\T^?$: For any $N$, let $\T_2(\Gamma_0(N))$ denote the full Hecke algebra acting on $S_2(\Gamma_0(N))$. The algebras $\T^\qunr$ and $\T$ are the completion of $\T_2(\Gamma_0(N))$ and  $\T_2(\Gamma_0(Nq))$, respectively, at the maximal ideal generated by $M:=\{p\}\cup\{T_l-a_l(f)\mid l\,{\!\not|} \,Nq\}$, i.e., the ideal corresponding to $\bar\rho$. If $q\not\equiv-1\pmod p$, then $\T^{\qnew}$ is the completion of  $\T_2(\Gamma_0(Nq))$ at the maximal ideal generated $\ffrm'$  by $M\cup\{U_q^2-1\}$, and if $q\equiv-1\pmod p$, then $\T^{\qnew}$ and $\T^{\qnewtw}$ are the completion of  $T_2(\Gamma_0(Nq))$ at the maximal ideal generated by $M$ together with either $U_q-1$ or $U_q+1$, respectively. By our hypothesis on $f$, the algebra $\T^\qunr$ is isomorphic to the ring of Witt vectors $W(k)$. By level raising theorems, the algebra $\T^{\qnew}$ (and $\T^{\qnewtw}$ for $q\equiv-1\pmod p$) has positive rank over $W(k)$.\medskip

The above means that the mod $p$ tangent space of $R^{\qunr}$ is trivial, i.e., that we have $H^1_{\Sc{L}^\qunr}(G_{\Q,S},\Ad^0(\bar{\rho}))=0$. The mod $p$ tangent spaces of $R^\qnew$ (or $R^\qnewtw$) are trivial if and only if the dimension of the space of (twisted) $q$-new forms is $1$. The Greenberg-Wiles formula then gives that
$$\dim H^1_{\Sc{L}}(G_{\Q,S},\Ad^0(\bar{\rho})) = \dim H^0(G_{\Q_q},\Ad^0(\bar{\rho})(1))=:\delta_q;$$ 
here $\delta_q=1$ if $q\not\equiv-1\pmod p$ and $\delta_q=2$ in the remaining case. Note that the strict positivity of $\delta_q$ is also implied by level raising which implies that the mod $p$ tangent space of $R$ is non-trivial. Let $k_0\subset k$ be the field of definition of $\Ad^0(\bar\rho)$. 
\begin{heuristic}\label{Heuristic-RibetsQ}
The restriction map
\[ r_q\colon H^1_{\Sc{L}}(G_{\Q,S},\Ad^0(\bar{\rho})) \longrightarrow  H^1(G_{\Q_q},\Ad^0(\bar{\rho})) \]
is random in $q$ as a map of $k_0$ vector spaces $k_0^{\delta_q}\to k_0^{\delta_q+1}$.
\end{heuristic}
It can be verified that the $k_0$ vector space on the right is the direct sum of the $1$-dimensional subspaces $\Sc{L}_q^\qunr$ and $\Sc{L}_q^\qnew$ (and also $\Sc{L}_q^\qnewtw$ if $q\equiv-1\pmod p$). We choose a corresponding basis of $H^1(G_{\Q_q},\Ad^0(\bar{\rho}))$. By the above we also know that $r_q^{-1}(\Sc{L}^\qunr)$ is zero. Hence with respect to a suitable basis also of $H^1_{\Sc{L}}(G_{\Q,S},\Ad^0(\bar{\rho}))$ we can write $r_q$ as a matrix 
\[
\left(\begin{array}{c}
a_q\\
1
\end{array}\right)
\quad\hbox{$\Bigg($ or}
\left(\begin{array}{cc}
a_q&b_q\\
1&0\\
0&1
\end{array}\right) \hbox{, respectively$\Bigg)$},
\]
for suitable $a_q$ (and $b_q$) in $k_0$. Now the dimension of $H^1_{\Sc{L}^\qnew}(G_{\Q,S},\Ad^0(\bar{\rho})) $ is that of the inverse image of $\Sc{L}^\qnew$ under $r_q$ (and similarly for $\Sc{L}^\qnewtw$). Hence its dimension is $0$ if $a_q\neq0$ and $1$ if $a_q=0$ (and similarly for $\Sc{L}^\qnew$ and $b_q$). The dimension is $0$ means that the $q$-new space of forms $f$ with $\bar\rho_f=\bar\rho$ has dimension one, or more precisely that $\T^\qnew[1/p]\cong W(k)[\frac1p]$ (or that $\T^\qnewtw[1/p]\cong W(k)[\frac1p]$).\medskip

Now the hypothesis that $q\to r_q$ behaves radomly simply means that the map $q\to a_q$ (and $q\to b_q$) takes random values in~$k_0$. Thus our heuristic predicts that the probability that the multiplicity of the $q$-new space(s) in question is larger than $1$ is $\frac1{\#k_0}$, and the probability is $1-\frac1{\#k_0}$ that the multiplicity is exactly~$1$. To be more explicit: 
\begin{proposition}\label{prop:RibetLR-1} Assume Heuristic~\ref{Heuristic-RibetsQ} holds, and suppose that $q\not\equiv-1\pmod p$ and that $\bar\rho(\Frob_q)$ is non-scalar.
Then $\dim S_2(\Gamma_0(Nq))^\qnew_{\I{m}'} = 1$ with density $1-\frac1{\#k_0}$. 
\end{proposition}
If $q\equiv-1\pmod p$, then level raising shows that there are $q$-new forms $f$ with $U_q=1$ and with $U_q=-1$ for which $\bar\rho_f=\bar\rho$. We decorate the corresponding spaces $S_2(\Gamma_0(Nq))^\qnew_{\I{m}'}$ with $U_q=1$ or $U_q=-1$ the exponent.
\begin{proposition}\label{prop:RibetLR-2}  Assume Heuristic~\ref{Heuristic-RibetsQ} holds, and suppose that $q\equiv-1\pmod p$.
Then for both signs $\eps\in\{\pm1\}$ we have $\dim S_2(\Gamma_0(Nq))^{\qnew,U_q=\eps}_{\I{m}'} = 1$ with density $1-\frac1{\#k_0}$.
\end{proposition}
One can also combine this, to get the prediction $\dim S_2(\Gamma_0(Nq))^\qnew_{\I{m}'} = 2$ with density $(1-\frac1{\#k_0})^2$ for $q\equiv-1\pmod p$.
\medskip

\subsection{Computational evidence}
We compared the predictions in Propositions~\ref{prop:RibetLR-1} and~\ref{prop:RibetLR-2} with numerically computed data for $p\in\{3,5,7,11\}$ using the computer algebra system Magma. Let $N=11$ if $p\in\{3,7,11\}$ and $N=17$ if $p=5$. Consider the unique ($\Q$-rational) normalized Hecke eigenform $f\in S_2(\Gamma_0(N))$. Here $N$ is chosen for $p$ so that $\bar\rho_f\colon G_\Q\to\GL_2(\F_p)$ is surjective. There is a corresponding elliptic curve $E_f$ over $\Q$ of~level~$N$. 
 \medskip

When $p$ is fixed, we distinguish four cases of primes $q$ not dividing $Np$. In each case we require $a_q(f)\equiv\pm(q+1)\pmod p$. We use indices i for $q\not\equiv\pm1\pmod p$, ii for $q\equiv-1\pmod p$, iii for $q\equiv1\pmod p$ and $\bar\rho_f(\Frob_q)$ non-scalar, and iv for $q\equiv1\pmod p$ and $\bar\rho_f(\Frob_q)$ scalar. For $?\in\{\textrm{i,ii,iii,iv}\}$ we denote by $S_?$ the set of all primes $q$ satisfying $?$. The sets $S_?$ are \v{C}ebotarov sets and analyzing conjugacy classes in $\GL_2(\F_p)$ gives
\[ \delta(S_{\textrm{i}}) = \frac{2(p-3)}{(p-1)^2},\ \delta(S_{\textrm{ii}}) = \frac{1}{(p-1)^2},\   \delta(S_{\textrm{iii}}) = \frac{2}{p^2+p},\  \delta(S_{\textrm{iv}}) = \frac{2}{(p^2-1)(p^2-p)}\]
for their densities.\medskip

This indicates a main problem with gathering numerical data: computing Hecke operators on $S_2(\Gamma_0(M))$ for $M\approx 10^6$ is already extremely slow, while the size of sets $S_?$ intersected with the interval $\{1..\lfloor 10^5/p\rfloor\}$ may be rather small for $?\in\{\textrm{ii,iii,iv}\}$. For instance if $p=11$ and we are in case iii, then $\delta(S_{\textrm{iii}})=\frac1{66}$, and one computes the cardinality of $S_{\textrm{iii}}\cap\{1,\ldots,10000\}$ to be $23$, and so one expects there to be only $2$-$3$ such $q$ for which $\dim S_2(\Gamma_0(Nq))^{\qnew}_{\I{m}'}>1$. Since the numerical tests suggest that cases with multiplicity higher than expected occur rather irregularly, having only $23$ samples is probably rather unreliable.\medskip

Another, perhaps less problematic, issue with the tables below is that they are not guaranteed to be correct. When computing the dimension of $S_2(\Gamma_0(Nq))^{\qnew}_{\I{m}'}$ we only intersected Hecke eigenspaces for $T_l-a_l(f)\pmod p$ for primes $l\neq N,p$ up to a certain $l$ that is much smaller than the Sturm bounded needed to have proven correctness of multiplicities -- if the dimension is larger than $1$. Our data is collected in Table~\ref{TableTwo}. 
\medskip

\renewcommand{\arraystretch}{1.3}
\begin{table}[htb]
\begin{center}
\advance\tabcolsep by .5em
\begin{tabular}{l|l|c|c|c|c|c|c|c|c|c|c|}
  &Case     & p=3&    p=5& p=7&  p=11\\ \hline
&\textrm{i} &   \textrm{---} & $\frac{99}{399}\approx.25$& $\frac{70}{503}\approx.14$
&$\frac{31}{321}\approx.097$\\ \cline{2-6}
Data&\textrm{ii}+ &    $\frac{145}{422}\approx.34$&$\frac{32}{124}\approx.26$&$\frac{11}{85}\approx.13$ 
&$\frac{3}{28}\approx.11$\\ \cline{2-6}
&\textrm{ii}$-$ &    $\frac{157}{422}\approx.37$&$\frac{26}{124}\approx.21$&$\frac{10}{85}\approx.12$ 
&$\frac{2}{28}\approx.071$\\ \cline{2-6}
&\textrm{iii} &   $ \frac{207}{630}\approx.33$
 &$\frac{41}{205}\approx.20$&$\frac{19}{150}\approx.13$
&$\frac{3}{78}\approx.038$\\ \hline 
\multicolumn{2}{l|}{Expected}  & $\frac{1}{3}\approx.33$ & $\frac{1}{5}=.2$  &  $\frac{1}{7}\approx.14$ & $\frac{1}{11}\approx.09$ \\ \hline
\end{tabular}
\medskip
\end{center}
\caption{}\label{TableTwo}
\end{table}

When displaying the output, we subdivide case ii in case ii+ and case ii$-$. This refers to the sign of the $U_q$-operator. For $p=3$ case i is empty. The data in case ? was produced by testing for fixed $p$ all $q$ up to a certain size, depending on $?$ and $p$. The denominator given in the rows labelled {\em Data} is the total number of $q$ in $S_?$ up to that bound. The numerator is the number of those $q$ where the multiplicity of the eigenspace in question is larger than expected.

\smallskip

In case iv we have no heuristic model. Here the data suggests $\dim S_2(\Gamma_0(Nq))^{\qnew}_{\I{m}'}\ge 4$, always. This was witnessed by $56$ primes $q$ for $p=3$, by $7$ primes $q$ for $p=5$, and by $2$ primes $q$ for $p=7$; in each case, the lower bound $4$ was attained most frequently. Case iv is rare in general, as is immediate from $\delta(S_{\textrm{iv}})$ displayed above. We found no $q$ for $p=11$ in case iv in our range of computation; the smallest candidate $q$ is $109297$. Let us mention that we were able to distinguish case iv from case iii because of \cite{Centeleghe} by Centeleghe and because our eigenforms $f$ correspond to elliptic curves defined over $\Q$. The computer code used is from~\cite{Centeleghe-Tsaknias}.

\begin{remark}
Let us also mention that R.\ Ramakrishna suggested the following refined version of Propositions \ref{prop:RibetLR-1} and \ref{prop:RibetLR-2}: Suppose that $q\not\equiv-1\pmod p$ and that $\bar\rho(\Frob_q)$ is non-scalar. Then for $i\ge1$ one has
\[\dim S_2(\Gamma_0(Nq))^\qnew_{\I{m}'} = i\quad \hbox{occurs with density \ } (\#k_0)^{1-i}(1-\frac1{\#k_0}).\]
The same statistics should also hold for the event $ \dim S_2(\Gamma_0(Nq))^{\qnew,U_q=\eps}_{\I{m}'} = i$ if $q\equiv-1\pmod p$ and $\eps$ is a sign in $\{\pm1\}$. 
In Table~\ref{TableThree} we compare our data with this prediction in the case $p=3$.
\begin{table}[htb]
\begin{center}
\begin{tabular}{l|*{13}{r|}}
multiplicity &all& 1 & 2 & 3& 4& 5& 10 \\ \hline
expected &&.66&.22&.074&.025&.008&.00003\\ \hline
$p=3$, case iii &504& 339&110&39&13&1&1\\ \hline
 ratio &&.67&.22&.077&.026&.0019&.0019\\ \hline
$p=3$, cases ii+ &319& 211 & 78&18&10&1&0\\ \hline
ratio &&.66&.24&.056&.031&.003&0\\ \hline
$p=3$, cases ii$-$ &319& 200 & 84&26&8&1&0\\ \hline
ratio &&.63&.26&.081&.025&.003&0\\ \hline
\end{tabular}
\end{center}\caption{}\label{TableThree}
\end{table}
\end{remark}
\medskip

\subsection{Level raising for $S_3$-representations}

We excluded above the case of dihedral representations as they exhibit different behavior. We illustrate this by discussing the case of odd $S_3$-representations $\rhobar: G_\Q \ra \GL_2(k)(=Aut(\overline W))$  with $k$ a fixed finite field of characteristic $p$ from the point of view of of level raising  at primes $q \neq p$. This is in parallel with the work in the pevious section which addressed weight one $p$-adic lifts of such $\rhobar$. 

We use the earlier notation and thus the determinant of $\rhobar$, the sign of the $S_3$-representation, is an odd character $\chi$.  Such a $\rhobar$ arises from  $S_2(\Gamma_1(Np),\chi \chi_p^{-1})$  for some $N$ prime to $p$, and we assume that $p>>0$ so there is a unique form in this space which fives rise to $\rhobar$ and thus the corresponding global Selmer group is 0. There are two cases of level raising at primes $q $ not 1 mod $p$, and unramified in $\rhobar$ and $\neq p$, to consider:

Case 1: $p$ is 1 mod $3$, $q$ splits in the splitting field of $\chi$, $\rhobar(\Frob_q)$ has eigenvalues $\omega, \omega^2$ with $\omega$ cube roots of unity in $k^*$ (which we may take to be the prime field $\F_p$), and $q$ is either $\omega$ or $\omega^{-1}$ mod $p$. In this case $\Ad^0=k(\chi)+ \overline W$, and locally at $q$, $\chi$ is trivial and $\overline W=k(\epsilon)+k(\epsilon^{-1})$. It is easy to see that  the  the $q$-new condition is the image of $H^1(G_q,k(\epsilon))$ in $H^1(G_q,\Ad^0)$, and  from this deduce that  the $q$-new Selmer group is always one-dimensional.

Case 2: We assume that $q$ is $ -1$ mod $p$, $p$ is 1 mod 3,  and $\rhobar(\Frob_q)$ is the image of complex conjugation. Locally at $q$, $k(\psi)=k(\varepsilon)$, and $\overline W=k+k(\varepsilon)$ where $\varepsilon $ is the mod $p$ cyclotomic charcater  of $G_q$ of order 2.  These eigenspaces correspond to the lines spanned by $H, E+F, E-F$ is the standard basis $H,E,F$ of $sl_2$. We have to choose which  twist of Steinberg we consider, and locally at  $q$, the $q$-new condition is either the image of
$H^1(G_q,k(H+E-F)(\varepsilon))$ or $H^1(G_q,k(H-E+F)(\varepsilon))$ in $H^1(G_q, \Ad^0)$. The   global $q$-new Selmer group  with given choice of Steinberg twist is 1-dimensional  if the  1-dimensional global Selmer group (with no restrictions at $q$) $H^1_{\mathcal L}(S \cup \{q\},\overline W)$, which is a summand of the $2$ dimensional  $H^1_{\mathcal L}(S \cup \{q\},\Ad^0)$, maps to the line $H^1(G_q,k(\epsilon))$ in the 2-dimensional local cohomology group $H^1(G_q,\overline W)$. Following our heuristics this may be expected to happen $1/p$ times (note that as $p$ is 1 mod 3, the field $k$ may be taken to be the prime field $\F_p$).


\begin{thebibliography}{CKPS}


\bibitem{AK}
Patrick Allen, Chandrashekhar Khare.
\newblock Weak Leopoldt conjecture for adjoint motives arising from regular algebraic cusp forms over CM fields.
\newblock preprint.


\bibitem{BLGGT} T. Barnet-Lamb, T. Gee, D. Geraghty, R. Taylor.
\newblock  Potential automorphy and change of weight.
\newblock  Ann. of Math. (2) 179 (2014), no. 2, 501--609.

\bibitem{BD}
Jo\"el Bella\"iche and M. Dimitrov.
\newblock On the eigencurve at classical weight 1 points.
\newblock Duke Math.\ J.\ 165:2 (2016), 245--266. 

\bibitem{BlCh}
F.\ Bleher, T.\ Chinburg.
\newblock Universal deformation rings need not be complete intersections. 
\newblock Math. Ann.\ 337 (2007), no.\ 4, 739--767.
 
\bibitem{Boston}
Nigel Boston.
\newblock Explicit deformation of Galois representations.
\newblock Invent.\ Math.\ 103 (1991), p. 181--196.

\bibitem{Boston-Ramif}
Nigel Boston.
\newblock Families of Galois representations -- increasing the ramification.
\newblock Duke Math.\ J.\ 66 (1992), no. 3, 357--367.




\bibitem{CG}
Frank Calegari and David Geraghty.
\newblock Modularity lifting beyond the Taylor-Wiles method.
\newblock preprint.

\bibitem{CHT} L. Clozel, M. Harris, R. Taylor.
\newblock  Automorphy for some l-adic lifts of automorphic mod l Galois representations. With Appendix A, summarizing unpublished work of Russ Mann, and Appendix B by Marie-France Vign\'eras.
\newblock Publ. Math. Inst. Hautes \'Etudes Sci. No. 108 (2008), 1--181.

\bibitem{Car}
Ana Caraiani.
\newblock  Local-global compatibility and the action of monodromy on nearby cycles. 
\newblock Duke Math. J. 161 (2012), no. 12, 2311--2413.

\bibitem{Centeleghe}
Tommaso Giorgio Centeleghe.
\newblock Integral Tate modules and splitting of primes in torsion fields of elliptic curves.
\newblock Int. J. Number Theory 12 (2016), no. 1, 237--248.

\bibitem{Centeleghe-Tsaknias}
Tommaso Giorgio Centeleghe and Panagiotis Tsaknias.
\newblock Integral Frobenius,
\newblock Magma  package  available  at http://math.uni.lu/$\sim$tsaknias/sfware.html

\bibitem{Cohen}
H.\ Cohen, 
\newblock A course in computational algebraic number theory. 
\newblock Graduate Texts in Mathematics, 138. Springer-Verlag, Berlin, 1993.

\bibitem{Colmez} Pierre Colmez.
\newblock R\'esidu en s=1 des fonctions zeta p-adiques. 
\newblock  Invent. Math. 91 (1988), no. 2, 371--389.

\bibitem{CDP}  R. Crandall, K. Dilcher, C. Pomerance, 
\newblock  A search for Wieferich and Wilson primes.
\newblock  Math. Comp. 66 (1997), no. 217, 433--449.

\bibitem{DW} 
\newblock Chantal David, Tom Weston.
\newblock  Local torsion on elliptic curves and the deformation theory of Galois representations. 
\newblock Math.\ Res.\ Lett.\ 15 (2008), no.\ 3, 599--611.

\bibitem{DFG} 
\newblock F.\ Diamond, M.\ Flach, L.\ Guo.
\newblock  The Tamagawa number conjecture of adjoint motives of modular forms.
\newblock Ann.\ Sci.\ \'Ecole Norm.\ Sup.\ (4) 37 (2004), no. 5, 663--727.
 
\bibitem{Gamzon} 
\newblock Adam Gamzon.
\newblock  Local torsion on abelian surfaces with real multiplication by $\Q(\sqrt{5})$.
\newblock Int.\ J.\ Number Theory 10 (2014), no.\ 7, 1807--1827. 
%
%


\bibitem{Gouvea}
Fernando Q.\ Gouv\^ea.
\newblock Deformations of Galois Representations.
\newblock in: Arithmetic Algebraic Geometry.
\newblock in: Ias/Park City Math. Ser., vol. 9, Amer. Math. Soc., Providence, RI, 2001, pp. 233--406.


\bibitem{Gras}
Georges Gras.
\newblock Les $\theta$-r\'egulateurs locaux d'un nombre alg\'ebrique: conjectures $p$-adiques.
\newblock Canadian Journal of Mathematics, Vol. 68, 3 (2016), 571--624.

\bibitem{Greenberg}
Ralph,  Greenberg.
\newblock Galois representations with open image. 
\newblock  Ann.\ Math.\ Qu\'e., 40 (2016), no. 1, 83--119. 

\bibitem{Guiraud}
David Guiraud.
\newblock A framework for unobstructedness of Galois deformation rings.
\newblock PhD Thesis, Universit\"at Heidelberg, 2015.


\bibitem{Hajir-Maire}
Farshid Hajir, Christian Maire.
\newblock Prime decomposition and the Iwasawa MU-invariant.
\newblock Math. Proc. Camb. Phil. Soc., Online, https://doi.org/10.1017/S0305004118000191

\bibitem{Jan}
Uwe Jannsen.
\newblock On the $l$-adic cohomology of varieties over number fields and its Galois cohomology,.
\newblock in: Galois groups over $\mathbb{Q}$ (Berkeley, CA, 1987).
\newblock MSRI Publ.\ 16, Springer, New York, 1989, pp. 315--360.

\bibitem{JKPSZ}
Bruce W. Jordan, Zev Klagsbrun, Bjorn Poonen, Christopher Skinner, Yevgeny Zaytman.
\newblock Statistics of $K$-groups modulo $p$ for the ring of integers of a varying quadratic number field.
\newblock arXiv:1703.00108.

\bibitem{Katz}
Nicholas M. Katz.
\newblock Wieferich past and future.
\newblock in: Cont.\ Math.: Proceedings of the 11th International Conference on Finite Fields, AMS, 2014.

\bibitem{Khare}
Chandrashekhar Khare.
\newblock Compatible systems of mod $p$ Galois representations and Hecke characters.
\newblock Math.\ Res.\ Lett.\ 10(1)  (2003), 71--83.



\bibitem{Maz}
Barry Mazur.
\newblock Deforming Galois representations.
\newblock MSRI Publication, Galois Groups over $\Q$.

\bibitem{Mazur-Fern}
Barry Mazur.
\newblock An "infinite fern" in the universal deformation space of Galois representations.
\newblock Collect.\ Math.\ 48 (1997), 155--193. 

\bibitem{Movahhedi-Nguyen}
Abbes Movahhedi, Thong Nguyen Quang Do.
\newblock Sur l'arithm\'etique des corps de nombres $p$-rationnels.
\newblock S\'eminaire de Th\'eorie des Nombres, Paris 1987-88, 155-200, Progr. Math. 81, Birkh\"auser Boston, Boston, MA, 1990. 

\bibitem{Ramakrishna-FM}
 Ravi Ramakrishna.
\newblock Deforming Galois representations and the conjectures of Serre and Fontaine-Mazur. 
\newblock Ann. of Math. (2) 156 (2002), no. 1, 115--154.

\bibitem{Ribet}
Ken Ribet.
\newblock  A Report on mod $l$ representations of $\Gal(\overline\Q/\Q)$. 
\newblock Motives (Seattle, WA, 1991), 639--676, Proc.\ Sympos.\ Pure Math.\, 55, Part 2, Amer.\ Math.\ Soc., Providence, RI, 1994.
%
%

\bibitem{Sprang}
J.\ Sprang, 
\newblock A universal deformation ring with unexpected Krull dimension.
\newblock Math.\ Z.\ 275 (2013), no. 1-2, 647--652.
 
 
\bibitem{Sil}
J. Silverman.
\newblock Wieferich's criterion and the $abc$ conjecture.
\newblock Journal of Number Theory 30 (1998), 226--237.

\bibitem{Tay}
Richard Taylor.
\newblock Automorphy for some $l$-adic lifts of automorphic mod $l$ Galois representations. II.
\newblock Publ. Math. Inst. Hautes \'Etudes Sci. No. 108 (2008), 183--239.

%
%
%
%

\bibitem{Wes}
Tom Weston. 
\newblock Unobstructed modular deformation problems.
\newblock  Amer. J. Math. 126 (2004), no. 6, 1237--1252.

%
%
%
\bibitem{Wiese}
Gabor Wiese. 
\newblock Multiplicities of Galois representations of weight one. With an appendix by Niko Naumann.
\newblock  Algebra Number Theory 1 (2007), no. 1, 67�85.

\end{thebibliography}
\end{document}